\documentclass[12pt,reqno]{amsart}      
\usepackage{geometry}
\usepackage{extarrows}
\usepackage{amssymb}
\usepackage{eucal}
\usepackage{bbm}
\usepackage{graphicx}
\usepackage{amsmath}
\usepackage{esint}
\usepackage{amscd}
\usepackage[all]{xy}           
\usepackage{amsfonts,latexsym}
\usepackage{xspace}
\usepackage{epsfig}
\usepackage{float}
\usepackage{color}
\usepackage{shuffle}
\usepackage{fancybox}
\usepackage{colordvi}
\usepackage{multicol}
\usepackage{colordvi}
\usepackage{mathrsfs}
\usepackage{ifpdf}
\ifpdf
  \usepackage[colorlinks,final,backref=page,hyperindex]{hyperref}
\else
  \usepackage[colorlinks,final,backref=page,hyperindex,hypertex]{hyperref}
\fi
\usepackage[active]{srcltx} 
\usepackage{tikz}
\DeclareMathAlphabet{\mathcal}{OMS}{cmsy}{m}{n}

\newtheorem{theorem}{Theorem}[section]
\newtheorem{proposition}[theorem]{Proposition}
\newtheorem{lemma}[theorem]{Lemma}
\newtheorem{coro}[theorem]{Corollary}
\newtheorem{prop-def}{Proposition-Definition}[section]
\newtheorem{coro-def}{Corollary-Definition}[section]

\theoremstyle{definition}
\newtheorem{definition}[theorem]{Definition}
\newtheorem{remark}[theorem]{Remark}

\newcommand{\nc}{\newcommand}
\nc{\tred}[1]{\textcolor{red}{#1}}
\nc{\tblue}[1]{\textcolor{blue}{#1}}
\nc{\tgreen}[1]{\textcolor{green}{#1}}
\nc{\tpurple}[1]{\textcolor{purple}{#1}}
\nc{\btred}[1]{\textcolor{red}{\bf #1}}
\nc{\btblue}[1]{\textcolor{blue}{\bf #1}}
\nc{\btgreen}[1]{\textcolor{green}{\bf #1}}
\nc{\btpurple}[1]{\textcolor{purple}{\bf #1}}
\nc{\NN}{{\mathbb N}}
\nc{\ncsha}{{\mbox{\cyr X}^{\mathrm NC}}} \nc{\ncshao}{{\mbox{\cyr
X}^{\mathrm NC}_0}}

\newcommand{\delete}[1]{}

\nc{\mlabel}[1]{\label{#1}}
\nc{\mcite}[1]{\cite{#1}}
\nc{\mref}[1]{\ref{#1}}
\nc{\meqref}[1]{\eqref{#1}}
\nc{\mbibitem}[1]{\bibitem{#1}}

\delete{
\nc{\mlabel}[1]{\label{#1}{\hfill \hspace{1cm}{\bf{{\ }\hfill(#1)}}}}
\nc{\mcite}[1]{\cite{#1}{{\bf{{\ }(#1)}}}}
\nc{\mref}[1]{\ref{#1}{{\bf{{\ }(#1)}}}}
\nc{\meqref}[1]{\eqref{#1}{{\bf{{\ }(#1)}}}}
\nc{\mbibitem}[1]{\bibitem[\bf #1]{#1}}
}

\nc{\xing}[1]{\textcolor{purple}{Xing:#1}}
\nc{\nan}[1]{\textcolor{blue}{Nan:#1}}

\numberwithin{equation}{section}
\nc{\etree}{1}
\nc{\RP}{{\mathcal{D}}^{\alpha}(\Delta_T, V)}
\nc{\WRP}{{\mathcal{D}}_{w}^{\alpha}(\Delta_T, V)}
\nc{\Y}{{\bf Y}}
\nc{\x}{\mathbb{X}}
\nc{\xx}{\mathcal{X}}
\nc{\ha}{\mathcal{H}}
\nc{\HA}{{\bf H}}

\nc{\RR}{\mathbb{R}} \nc{\ZZ}{\mathbb{Z}} \nc{\V}{\RR^{d}} \nc{\pro}{\otimes}
\nc{\tng}{T^{\le N}(\RR^d)^{g}} \nc{\tn}{T^{\le N}(\RR^d)}
\nc{\ttg}{T^{\le 3}(\RR^d)^{g}}
\nc{\X}{{\bf X}} \nc{\Z}{{\bf Z}} \nc{\W}{{\bf W}}
\nc{\E}{{\bf E}} \nc{\J}{{\bf J}}
\nc{\DA}{\mathcal{D}^{\alpha}([0,2\pi], \RR^d)}
\nc{\C}{\mathcal{C}^{\alpha}}
\nc{\D}{\mathcal{D}^{\alpha}(\Delta_T, V)}
\nc{\CC}{\mathcal{C}_{\X}^{\alpha}} \nc{\CCA}{\mathcal{C}_{\widetilde \X}^{\alpha}}
\nc{\WCC}{\mathcal{C}_{w, \X}^{\alpha}}

\nc{\f}{\varphi}
\nc{\al}{\alpha}
\nc{\lbar}{\overline}
\nc{\Hh}{{\mathfrak{H}}}
\nc{\id}{\text{id}} \nc{\Id}{\text{Id}}
\nc{\bx}{(1, X, \x, \xx)}
\nc{\bh}{(\delta H_t, \HA_t, \ha_t)}

\begin{document}
\title[Rough Differential Equations Driven by Controlled Rough Paths]{A Direct Approach to Rough Differential Equations Driven by Controlled Rough Paths}
\author{Nannan Li}
\address{School of Mathematics and Statistics, Lanzhou University
Lanzhou, 730000, China
}
\email{linn2024@lzu.edu.cn}

\author{Xing Gao$^{*}$}\thanks{*Corresponding author}
\address{School of Mathematics and Statistics, Lanzhou University
Lanzhou, 730000, China;
Gansu Provincial Research Center for Basic Disciplines of Mathematics
and Statistics, Lanzhou, 730070, China
}
\email{gaoxing@lzu.edu.cn}

\begin{abstract}
We study rough differential equations driven by controlled rough paths. 
Given a rough path $\mathbf X$ and an $\mathbf X$-controlled rough path $\mathbf Z$, we develop a direct controlled rough paths approach to the rough differential equation $dY_t=F(Y_t)\,d\mathbf Z_t$. 
More specifically, we establish rough integral estimates and the universal limit theorem directly in the space of $\X$-controlled rough paths.
\end{abstract}

\makeatletter
\@namedef{subjclassname@2020}{\textup{2020} Mathematics Subject Classification}
\makeatother
\subjclass[2020]{
60L20, 
60L50, 
60H15, 
60H17. 
}

\keywords{rough paths, controlled rough path, rough integral, universal limit theorem}

\maketitle

\tableofcontents

\setcounter{section}{0}

\allowdisplaybreaks

\section{Introduction}
Rough path theory, initiated by Lyons~\cite{Ly98}, provides a deterministic framework for differential equations driven by irregular signals. 
In the classical setting, one considers
\begin{equation}\label{eq:I1}
dY_t=F(Y_t)\,dX_t,
\end{equation}
where the driving path $X$ is enhanced to a rough path $\mathbf X$ carrying the iterated-integral information required to interpret the equation.
This enhancement makes the solution map pathwise well posed and continuous in the rough path topology, a statement known as universal limit theorem~\cite{FH20,FV10,LGM26, Ly98,LQ02}. 
This pathwise stability also provides a natural framework for understanding Wong--Zakai approximations~\cite{WZ1,WZ2}.

An alternative but closely related viewpoint is provided by Gubinelli's theory of controlled rough paths~\cite{Gu04}. 
In this framework, one considers paths whose increments admit local expansions in terms of a reference rough path. 
Controlled rough paths arise naturally in rough integration and rough differential equations, and their structure is preserved under suitable compositions and integrations. 
In particular, the solution of a rough differential equation driven by a rough path may itself serve as the driving signal of another rough differential equation.
This leads naturally to rough differential equations whose driving signals are controlled rough paths.

The purpose of the present paper is to study such rough differential equations directly within the controlled rough path framework. 
More precisely, rather than first passing to a rough path enhancement of the effective driving signal and then applying classical rough differential equation theory, we establish the corresponding integration estimates, existence and uniqueness results, and stability estimates in the controlled rough path space.

\subsection{Controlled rough integral}\label{sec:1}
For arbitrary rough path $\mathbf X=(1, X, \mathbb X)$, Gubinelli~\cite{Gu04} introduced the rough integral of one controlled path against another through compensated Riemann sums:
\begin{equation}\label{eq:I2}
\int_s^t Y_r\,d\mathbf Z_r:=\lim_{|P|\to0}\sum_{[u,v]\in P}\bigl(Y_uZ_{u,v}+Y'_uZ'_u\mathbb X_{u,v}\bigr),
\end{equation}
where $\mathbf Y=(Y,Y')$ and $\mathbf Z=(Z,Z')$ are $\mathbf X$-controlled paths.
For convenience, we refer to this type of rough integral as the controlled rough integral.
The classical rough integral with respect to $\mathbf X$ is recovered by taking $\mathbf Z=(X,\mathrm{id})$. 
Its construction and fundamental estimates are well established; see Gubinelli~\cite{Gu04}, Friz and Hairer~\cite{FH20}.

An important property of the controlled rough integral is that its indefinite integral is again a controlled rough path.
More precisely, if
\[
I_t=\int_0^t Y_r\,d\mathbf Z_r,
\]
then the corresponding Gubinelli derivative is $I'=YZ'$. 
Consequently, controlled rough integral can be performed within the controlled rough path space. 
This property, together with suitable quantitative estimates, provides the basic analytical framework for studying rough  differential equations driven by controlled rough paths.

The equation~\eqref{eq:I2} is particularly relevant to systems with a layered structure, where the driving signal of one equation is generated by another rough system. 
For instance, if a lower-level equation driven by $\X$ produces an output $\Z$, then $\Z$ is typically controlled by $\X$, and an upper-level equation should be written directly in terms of $d\Z$. 
This keeps the layered structure visible and avoids unnecessarily enlarging the state space. 
The analytical framework developed here is related to several existing approaches, including Young integration~\cite{You36}, the sewing lemma~\cite{BZ22,FdLP06,Gu04,Le20}, and Hopf-algebraic controlled rough paths~\cite{ZGLM25}. 

In this paper, we revisit the integral~\eqref{eq:I2} using a point-removal argument and derive the associated remainder estimate in the controlled rough path norms used throughout our analysis. 
The point-removal method and the underlying integration estimates are classical; see, for example, Friz and Hairer~\cite[Lemma 4.2]{FH20}. 
Our purpose is to present these estimates in a form directly suited to the fixed point theorem and stability analysis.

\subsection{Controlled RDEs}\label{sec:2}
We consider rough differential equations driven by controlled rough paths of the form
\begin{equation}\label{eq:I3}
 dY_t=F(Y_t)\,d\Z_t,
\end{equation}
where $\Z=(Z,Z')$ is an $\X$-controlled rough path. 
We refer to~\eqref{eq:I3} as a Controlled RDE. 
The integral is understood in the sense of~\eqref{eq:I2}.

The natural controlled structure of a solution $\Y=(Y,Y')$ is determined by $Y'_t=F(Y_t)Z'_t$.
Indeed, the closure of controlled rough paths under composition and integration allows us to consider the mapping
\[
\mathcal M(Y, Y')=\left(Y_0+\int_0^\cdot F(Y_r)\,d\Z_r,\,F(Y)Z'\right).
\]
This mapping is well defined on the appropriate controlled path space. 
In particular, the input $\Y=(Y, Y')$ is only required to be $\X$-controlled rough path; it need not be controlled by the driving signal $Z$. 
Hence the existence and uniqueness of solutions can be established directly in the $\X$-controlled rough path space, without constructing a rough path enhancement of $Z$.

This direct approach is compatible with classical rough path theory.
Given an $\X$-controlled rough path $\Z$, the standard theory provides a canonical rough path enhancement of $Z$ and identifies the corresponding controlled rough integral with rough integral against the induced rough path; see~\cite[Proposition 7.1]{FH20}. 
Consequently, the Controlled RDEs can also be treated as classical RDEs driven by this enhancement. 
The distinction considered here is therefore not a different class of solutions, but a direct analytical approach based on the controlled structure of the driving signal.

Besides existence and uniqueness, an essential aspect of the theory is the continuous dependence of solutions on the driving data. For~\eqref{eq:I3}, the natural data consist of the initial condition $Y_0$, the reference rough path $\mathbf X$, and the controlled driver $\mathbf Z$. 
We therefore study the solution map
\[
(\mathbf X,\mathbf Z,Y_0)\longmapsto\mathbf Y
\]
and establish stability estimates with respect to perturbations of these data. 

Such estimates can also be related to the classical universal limit theorem through the continuity of the canonical lift and the stability of classical RDEs~\cite{BG22,FH20,Gu04}. 
Our approach instead derives the required estimates directly from controlled rough integration and composition, so that the entire stability analysis is carried out in terms of the controlled rough paths. 
This provides a unified treatment of the Controlled RDEs and is particularly convenient when the effective driving signal arises from another rough system.\\

\noindent\textbf{Organization of the paper.}
Section~\ref{ss:sec2} recalls controlled rough integration and develops the quantitative estimates needed later. Section~\ref{ss:sec3} establishes the local and global well-posedness of Controlled RDEs. 
Section~\ref{ss:sec4} proves the stability estimates and the universal limit theorem in the controlled rough paths framework.\\

\noindent\textbf{Notation.}
Throughout this paper, we work over the field $\RR$, which serves as the base field for all vector spaces, tensor products, algebras, coalgebras, and linear maps under consideration. 
Unless otherwise stated, all tensor products of Banach spaces are equipped with a fixed admissible cross norm and then completed. In particular, for Banach spaces $U$ and $V$, the notation $U\otimes V$ denotes the completed tensor product with respect to this norm. 
For definiteness, one may take the projective tensor norm throughout. 
The notation $M(\bullet)$ denotes a universal function, increasing in all of its arguments, whose value may change from line to line.

Let $U$, $V$ and $W$ be Banach spaces. We use $V$ as the underlying space of the driving rough path (Definition~\mref{def:wrp}) and $W$ as the value space of controlled paths (Definition~\mref{def:crp}). 
The space $U$ is the codomain of the rough integral in Lemma~\mref{thm:1}. For a continuous path
\[
X:[0,T]\to V,\qquad t\mapsto X_t,
\]
we denote its increment over an interval $[s,t]$ by $X_{s,t}:=X_t-X_s$.

\section{Preliminaries}\label{ss:sec2}
In this section, we first recall the notions of rough paths and controlled rough paths. 
We then introduce the controlled rough integral and present the corresponding estimate results.

\begin{definition}\cite{BG22, FH20}
Let $V$ be a Banach space, $\alpha\in(0,1]$ and $N:=\lfloor 1/\alpha\rfloor$. Set $\Delta_T:=\{(s,t) \mid 0\leq s\leq t \leq T\}$. An {\bf $\alpha$-H\"older rough path} is a map
\[
\X=(X^0, X^1, \ldots, X^N): \Delta_T\to \bigoplus_{k=0}^N V^{\otimes k}
\]
such that:
\begin{enumerate}
\item[\rm (1)] {\bf (Chen's relation)} For all $0\le s\le u\le t\le T$,
\[
\X_{s,u}\otimes \X_{u,t}=\X_{s,t}.
\]

\item[\rm (2)] {\bf (Analytic estimate)} For each $i=0,\ldots,N$,
\[
\|X^i\|_{i\alpha}:= \sup_{s\neq t\in[0,T]}\frac{|X^i_{s,t}|}{|t-s|^{i\alpha}}<\infty.
\]
\end{enumerate}
\mlabel{def:wrp}
\end{definition}

Note that the zero-th component $X^{0}:\Delta_T\to \mathbb{R}$  must be the constant map equal to $1$. 
Accordingly, throughout this paper we write $X^{0}=1$.
Let $\RP$ denote the set of $\alpha$-H\"older rough paths. For each $\X\in \RP$, define
\begin{equation}
\|\X\|_{\alpha}:=\sum_{i=1}^{N}\|X^i\|_{i\alpha}.
\mlabel{eq:note4}
\end{equation}
For two rough paths 
\[
\X=(X^0, X^1, \ldots, X^N), \quad \widetilde{\X}=(\widetilde X^0, \widetilde X^1,\ldots, \widetilde X^N),
\]
we use the distance
\[
\|\X-\widetilde{\X}\|_\alpha:=\sum_{i=1}^{N}\|X^i-\widetilde X^i\|_{i\alpha}.
\]

We now recall the definition of controlled rough paths. \vspace{0.1in}

\begin{definition}~\cite{BG22, FH20}
Let $V$, $W$ be Banach spaces, and let $Y:[0,T]\to W$ be a path. Fix $\alpha\in(0,1]$, and set $N=\lfloor 1/\alpha\rfloor$. Let $\X=(1,X^1,\ldots,X^N)\in\D$. A collection
\[
\Y=(Y^0,\ldots,Y^{N-1}):[0,T]\to W\times \mathcal L(V,W)\times\cdots\times \mathcal L\bigl(V^{\otimes (N-1)},W\bigr)
\]
is called an {\bf $\X$-controlled rough path} if and only if
\begin{equation}
\|R^{i,\mathbf Y}\|_{(N-i)\alpha} :=\sup_{s\neq t\in[0,T]}\frac{|R^{i,\mathbf Y}_{s,t}|}{|t-s|^{(N-i)\alpha}} <\infty,\quad i=0,\ldots,N-1,
\mlabel{eq:remai}
\end{equation}
where
\begin{equation}
R^{i,\mathbf Y}_{s,t}:=
\begin{cases}
Y_t^i-Y_s^i- \sum\limits_{j=1}^{N-1-i}Y_s^{i+j}X^j_{s,t}, & i=0,\ldots,N-2,\\[1ex]
Y_t^{N-1}-Y_s^{N-1}, & i=N-1.
\end{cases}
\mlabel{eq:-1}
\end{equation}
\mlabel{def:crp}
\end{definition}

For $\X\in \RP$, denote by $\CC([0,T], W)$ the space of $\X$-controlled rough paths.
Define a norm $|||\cdot|||_{\X;\al}$ on $\CC([0,T], W)$ as
\begin{equation*}
|||\Y|||_{\X; \al}:=\sum_{i=0}^{N-1}|Y_0^i|+\sum_{i=0}^{N-1}\|R^{i, \Y}\|_{(N-i)\al}.
\end{equation*}
Note that the normed space $\Big(\CC([0,T], W),  |||\cdot|||_{\X;\al}\Big)$ is a Banach space~\cite{BG22}. 
For the purpose of calculation, we also need a seminorm $\|\cdot\|_{\X;\al}$ on $\CC([0,T], W)$~\cite{BG22}, which is given by
\begin{equation}
\|\Y\|_{\X; \al}:=\sum_{i=0}^{N-1}\|R^{i, \Y}\|_{(N-i)\al}.
\mlabel{eq:norm2}
\end{equation}

For continuity estimates, it is also important to quantify the distance between controlled rough paths defined with respect to different driving rough paths. 
Let $\mathbf X$ and $\widetilde{\mathbf X}$ be $\alpha$-H\"older rough paths, and let $\Y\in \CC([0,T], W)$ and $\widetilde \Y\in \mathcal{C}_{\widetilde \X}^{\alpha}([0,T], W)$. Define the distance~\cite{BG22}
\begin{equation}
d_{\X, \widetilde \X; \al}(\Y, \widetilde \Y):=\sum_{i=0}^{N-1}\|R^{i, \Y}- R^{i, \widetilde \Y}\|_{(N-i)\al}.
\mlabel{eq:norm1}
\end{equation}

The controlled rough integral was originally introduced in~\cite{Gu04}. 
Here we provide a point-removal proof and derive the corresponding estimate for use in the subsequent analysis.
\smallskip

\begin{lemma}~\cite[Theorem 1]{Gu04}
Let $\al \in(\frac{1}{3},\frac{1}{2}]$ and $\X =(1, X, \x)\in \D$. Let $\Y = (Y, Y')\in \CC([0, T], \mathcal{L}(W, U))$ and $\Z = (Z, Z')\in \CC([0, T], W)$. For  $0\le s<t\le T$, the controlled rough integral
\begin{equation*}
\int_s^t Y_r\,d\Z_r:=\lim_{|\pi | \to 0} \sum_{[u, v]\in \pi}Y_u  Z_{u,v}+Y'_uZ'_{u} \x_{u,v}\in U
\end{equation*}
is well-defined, where $\pi$ is an arbitrary partition of $[s, t]$.
\mlabel{thm:1}
\end{lemma}

\begin{remark}
Here we explain the notation $Y'_u Z'_u\,\x_{u,v}$. Suppose $\x_{u,v}=x_1\otimes x_2$, then this notation $Y'_u Z'_u\,\x_{u,v}$ can be interpreted as
\[Y'_u Z'_u\,\x_{u,v}:=[Y'_u(x_1)]\big(Z'_u(x_2) \big)\in U,\]
where $Y'_u(x_1):W \to U$ and $Z'_u(x_2)\in W$.
\end{remark}

\begin{proof}
Let 
$$P:s=t_0<t_1<\cdots<t_{n-1}<t_n=t$$
be a finite partition of $[s, t]$. For $j = 1, \ldots , n-1$, $P\setminus \{t_j\}$ denotes the partition obtained by removing the interior
point $t_j$.
Define the enhanced Riemann-sum approximation by
\[\int_PY_r\,d\Z_r:=\sum_{i=0}^{n-1}Y_{t_i} Z_{t_i, t_{i+1}}+Y'_{t_i}Z'_{t_i}\x_{t_i, t_{i+1}}.\]
Then
\begin{equation}
\begin{aligned}
&\ \hspace{-1.2cm}\int_PY_r\,d\Z_r-\int_{P\setminus \{t_j\}}Y_r\,d\Z_r\\
=&\ (Y_{t_{j-1}}Z_{t_{j-1}, t_{j}}+Y'_{t_{j-1}}Z'_{t_{j-1}} \x_{t_{j-1}, t_{j}})+(Y_{t_{j}}Z_{t_{j}, t_{j+1}}+Y'_{t_{j}}Z'_{t_{j}} \x_{t_{j}, t_{j+1}}) \\
&\ -(Y_{t_{j-1}}Z_{t_{j-1}, t_{j+1}}+Y'_{t_{j-1}}Z'_{t_{j-1}} \x_{t_{j-1}, t_{j+1}})\\
=&\ Y_{t_{j-1}}(Z_{t_{j-1}, t_{j}}-Z_{t_{j-1}, t_{j+1}})+Y'_{t_{j-1}}Z'_{t_{j-1}} (\x_{t_{j-1}, t_{j}}-\x_{t_{j-1}, t_{j+1}})\\
&\ +(Y_{t_{j}}Z_{t_{j}, t_{j+1}}+Y'_{t_{j}}Z'_{t_{j}} \x_{t_{j}, t_{j+1}})\\
=&\ -Y_{t_{j-1}}Z_{t_{j}, t_{j+1}}-Y'_{t_{j-1}}Z'_{t_{j-1}}(\x_{t_{j}, t_{j+1}}+X_{t_{j-1}, t_{j}}\otimes X_{t_{j}, t_{j+1}}) \\
&\ +(Y_{t_{j}}Z_{t_{j}, t_{j+1}}+Y'_{t_{j}}Z'_{t_{j}} \x_{t_{j}, t_{j+1}}) \\
=&\ Y_{t_{j-1}, t_{j}}Z_{t_{j}, t_{j+1}}+(Y'_{t_{j}}Z'_{t_{j}}-Y'_{t_{j-1}}Z'_{t_{j-1}})\x_{t_{j}, t_{j+1}}-Y'_{t_{j-1}}Z'_{t_{j-1}}(X_{t_{j-1}, t_{j}}\otimes X_{t_{j}, t_{j+1}}).
\end{aligned}
\mlabel{eq:20}
\end{equation}
By~(\ref{eq:-1}), we have
\[Y_{s,t}= R^{0, \Y}_{s,t}+ Y'_sX_{s,t},\quad Z_{s,t}= R^{0, \Z}_{s,t}+ Z'_sX_{s,t}.\]
Substituting the above equation into~(\ref{eq:20}) yields
\begin{equation}
\begin{aligned}
&\ \hspace{-2.9cm}\int_PY_r\,d\Z_r-\int_{P\setminus \{t_j\}}Y_r\,d\Z_r\\
=&\ R^{0, \Y}_{t_{j-1}, t_{j}}R^{0, \Z}_{t_{j}, t_{j+1}}+R^{0, \Y}_{t_{j-1}, t_{j}}Z'_{t_{j}}X_{t_{j}, t_{j+1}}\\
&\ +Y'_{t_{j-1}}X_{t_{j-1}, t_{j}}R^{0, \Z}_{t_{j}, t_{j+1}}+Y'_{t_{j-1}}Z'_{t_{j}}X_{t_{j-1}, t_{j}}\otimes X_{t_{j}, t_{j+1}}\\
&\ +(Y'_{t_{j}}Z'_{t_{j}}-Y'_{t_{j-1}}Z'_{t_{j-1}})\x_{t_{j}, t_{j+1}}-Y'_{t_{j-1}}Z'_{t_{j-1}}(X_{t_{j-1}, t_{j}}\otimes X_{t_{j}, t_{j+1}})\\
=&\ R^{0, \Y}_{t_{j-1}, t_{j}}R^{0, \Z}_{t_{j}, t_{j+1}}+R^{0, \Y}_{t_{j-1}, t_{j}}Z'_{t_{j}}X_{t_{j}, t_{j+1}}\\
&\ +Y'_{t_{j-1}}X_{t_{j-1}, t_{j}}R^{0, \Z}_{t_{j}, t_{j+1}}+Y'_{t_{j-1}}Z'_{t_{j-1}, t_{j}}X_{t_{j-1}, t_{j}}\otimes X_{t_{j}, t_{j+1}}\\
&\ +\Big((Y'_{t_{j}}-Y'_{t_{j-1}})Z'_{t_{j}}+Y'_{t_{j-1}}(Z'_{t_{j}}-Z'_{t_{j-1}})\Big)\x_{t_{j}, t_{j+1}} .
\end{aligned}
\mlabel{eq:20+1}
\end{equation}
Hence
\begin{equation}
\begin{aligned}
\left|\int_PY_r\,d\Z_r-\int_{P\setminus \{t_j\}}Y_r\,d\Z_r\right|
\le&\ \|R^{0, \Y}\|_{2\al}\|R^{0, \Z}\|_{2\al}|t_{j}-t_{j-1}|^{2\al}|t_{j+1}-t_{j}|^{2\al}\\
&\ +\|R^{0, \Y}\|_{2\al}\|Z'\|_{\infty}\|X\|_{\al}|t_{j}-t_{j-1}|^{2\al}|t_{j+1}-t_{j}|^{\al}\\
&\ +\|R^{0, \Z}\|_{2\al}\|Y'\|_{\infty}\|X\|_{\al}|t_{j}-t_{j-1}|^{\al}|t_{j+1}-t_{j}|^{2\al}\\
&\ +\|Y'\|_{\infty}\|Z'\|_{\al}\|X\|_{\al}^2|t_{j}-t_{j-1}|^{2\al}|t_{j+1}-t_{j}|^{\al}\\
&\ +\|Y'\|_{\al}\|Z'\|_{\infty}\|\x\|_{2\al}|t_{j}-t_{j-1}|^{\al}|t_{j+1}-t_{j}|^{2\al}\\
&\ +\|Y'\|_{\infty}\|Z'\|_{\al}\|\x\|_{2\al}|t_{j}
-t_{j-1}|^{\al}|t_{j+1}-t_{j}|^{2\al}.
\end{aligned}
\mlabel{eq:5}
\end{equation}

We aim to estimate~\eqref{eq:5} in terms of $\|\mathbf Y\|_{\mathbf X;\alpha}$ and $\|\mathbf Z\|_{\mathbf X;\alpha}$, which do not include the quantities
\[
\|Y'\|_{\infty}\, \text{ and }\,  \|Z'\|_{\infty}.
\]
Therefore first control these terms. 
For $\|Y'\|_{\infty}$, we have
\begin{equation}
\begin{aligned}
\|Y'\|_{\infty}=\sup_{0\le u\le T}|Y'_u|
&\ \le \sup_{0\le u\le T}|Y'_u-Y'_0|+|Y'_0|\\
&\ \le T^{\al}\sup_{0\le u\le T}{|Y'_u-Y'_0|\over |u-0|^{\al}}+|Y'_0|\\
&\ \le T^{\al}\|Y'\|_{\al}+|Y'_0|. 
\end{aligned}
\mlabel{eq:6}
\end{equation}
Similarly,
\begin{align}
\|Z'\|_{\infty}\le T^{\al}\|Z'\|_{\al}+|Z'_0|. \mlabel{eq:13}
\end{align}
Substituting~(\ref{eq:6}) and~(\ref{eq:13}) into~(\ref{eq:5}),
\begin{equation}
\begin{aligned}
&\ \hspace{-1.8cm}\left|\int_PY_r\,d\Z_r-\int_{P\setminus \{t_j\}}Y_r\,d\Z_r\right|\\
\le&\ \Big(1+T^{\al}+T^{2\al}\Big)\Big(|Y'_0|+\|R^{0, \Y}\|_{2\al}+\|R^{1, \Y}\|_{\al}\Big)\\
&\ \times\Big(|Z'_0|+\|R^{0, \Z}\|_{2\al}+\|R^{1, \Z}\|_{\al}\Big)\Big(1+\|X\|_{\al}+\|\x\|_{2\al}\Big)^2|t_{j+1}-t_{j-1}|^{3\al}\\
=&\ \Big(1+T^{\al}+T^{2\al}\Big)\Big(|Y'_0|+\|\Y\|_{\X; \al}\Big)\\
&\ \times\Big(|Z'_0|+\|\Z\|_{\X; \al}\Big)\Big(1+\|\X\|_{\al}\Big)^2|t_{j+1}-t_{j-1}|^{3\al} \quad (\text{by (\ref{eq:note4}) and (\ref{eq:norm2})}).
\end{aligned}
\mlabel{eq:14}
\end{equation}
Since
\begin{equation*}
\sum_{j=1}^{n-1}(t_{j+1}-t_{j-1})=(t_n-t_1)+(t_{n-1}-t_0)\le 2(t-s),
\end{equation*}
there exists an index $j$ such that
\[t_{j+1}-t_{j-1}\le \frac{2(t-s)}{n-1},\]
whence
\begin{align*}
&\ \hspace{-1.95cm}\left|\int_PY_r\,d\Z_r-\int_{P\setminus \{t_j\}}Y_r\,d\Z_r\right|\\
\le&\ \Big(1+T^{\al}+T^{2\al}\Big)(|Y'_0|+\|\Y\|_{\X; \al})(|Z'_0|+\|\Z\|_{\X; \al})(1+\|\X\|_{\al})^2\Big(\frac{2(t-s)}{n-1}\Big)^{3\al}.
\end{align*}
Now the new partition $P\setminus \{t_j\}$ has one point less than the original one and we continue this procedure until all points in $P$ (except for the endpoints $s$, $t$) are removed. 
A direct application of the triangle inequality yields the desired conclusion:
\begin{equation}
\begin{aligned}
&\ \hspace{-3cm}\left|\int_PY_r\,d\Z_r-\int_{\{s, t\}}Y_r\,d\Z_r\right|\\
\le&\ \Big(1+T^{\al}+T^{2\al}\Big)(|Y'_0|+\|\Y\|_{\X;\al})(|Z'_0|+\|\Z\|_{\X;\al})\\
&\ \times (1+\|\X\|_{\al})\big(2(t-s)\big)^{3\al}\Big({1\over (n-1)^{3\al}}+{1\over (n-2)^{3\al}}+\cdots+1\Big)\\
\le&\ C_{\al, T}(|Y'_0|+\|\Y\|_{\X; \al})(|Z'_0|+\|\Z\|_{\X; \al})(1+\|\X\|_{\al})^2|t-s|^{3\al}, 
\end{aligned}
\mlabel{eq:15}
\end{equation}
where
\[C_{\al, T}:=(1+|T|^{\al}+|T|^{2\al})2^{3\al}\sum_{n=1}^{\infty}{1\over n^{3\al}}<\infty \hspace{1cm} (\text{by $3\al>1$}).\]

To prove the convergence of Riemann-sum approximation, let $P$ and $\widetilde{P}$ be two given partitions of the interval $[s, t]$.
Let $P'$ be the partition formed by all points in $P$ and $\widetilde{P}$. For each sub-interval $[s_l, s_{l+1}]$ in $P$, the estimate~(\ref{eq:15}) implies
\begin{equation}
\begin{aligned}
&\ \hspace{-2.65cm}\left|\int_{P'\cap [s_l, s_{l+1}]}Y_r\,d\Z_r-\int_{\{s_l, s_{l+1}\}}Y_r\,d\Z_r\right|\\
&\ \le C_{\al, T}(|Y'_0|+\|\Y\|_{\X; \al})(|Z'_0|+\|\Z\|_{\X; \al})(1+\|\X\|_{\al})^2|s_{l+1}-s_l|^{3\al}.
\end{aligned}
\mlabel{eq:16}
\end{equation}
By summing over $l$,
\begin{equation}
\begin{aligned}
&\ \hspace{-0.85cm}\left|\int_{P'}Y_r\,d\Z_r-\int_{P}Y_r\,d\Z_r\right|\\
\le&\ \sum_{s_l\in P}\left|\int_{P'\cap [s_l, s_{l+1}]}Y_r\,d\Z_r-\int_{\{s_l, s_{l+1}\}}Y_r\,d\Z_r\right|\\
\le&\ C_{\al, T}(|Y'_0|+\|\Y\|_{\X; \al})(|Z'_0|+\|\Z\|_{\X; \al})(1+\|\X\|_{\al})^2\Big(\sum_{s_l\in P}(s_{l+1}-s_l)^{3\al}\Big)\quad (\text{(\ref{eq:16})})\\
=&\ C_{\al, T}(|Y'_0|+\|\Y\|_{\X; \al})(|Z'_0|+\|\Z\|_{\X; \al})(1+\|\X\|_{\al})^2\Big(\sum_{s_l\in P}(s_{l+1}-s_l)^{3\al-1}(s_{l+1}-s_l)\Big) \\
\le&\ C_{\al, T}(|Y'_0|+\|\Y\|_{\X; \al})(|Z'_0|+\|\Z\|_{\X; \al})(1+\|\X\|_{\al})^2\Big(\sum_{s_l\in P}|P|^{3\al-1}(s_{l+1}-s_l)\Big)\\
=&\ C_{\al, T}(|Y'_0|+\|\Y\|_{\X; \al})(|Z'_0|+\|\Z\|_{\X; \al})(1+\|\X\|_{\al})^2|P|^{3\al-1}\sum_{s_l\in P}(s_{l+1}-s_l)\\
=&\ C_{\al, T}(|Y'_0|+\|\Y\|_{\X; \al})(|Z'_0|+\|\Z\|_{\X; \al})(1+\|\X\|_{\al})^2|P|^{3\al-1}|t-s|.
\end{aligned}
\mlabel{eq:17}
\end{equation}
A similar estimate holds with $P$ replaced by $\widetilde{P}$:
\begin{equation}
\begin{aligned}
&\ \hspace{-2.5cm}\left|\int_{P'}Y_r\,d\Z_r-\int_{\widetilde{P}}Y_r\,d\Z_r\right|\\
&\ \le C_{\al, T}(|Y'_0|+\|\Y\|_{\X; \al})(|Z'_0|+\|\Z\|_{\X; \al})(1+\|\X\|_{\al})^2|\widetilde{P}|^{3\al-1}|t-s|.
\end{aligned}
\mlabel{eq:18}
\end{equation}
Combining~(\ref{eq:17}) and~(\ref{eq:18}),
\begin{align*}
&\ \hspace{-1.4cm}\left|\int_{\widetilde{P}}Y_r\,d\Z_r-\int_{P}Y_r\,d\Z_r\right|\\
&\ =\left|\left(\int_{\widetilde{P}}Y_r\,d\Z_r-\int_{P'}Y_r\,d\Z_r\right)+\left(\int_{P'}Y_r\,d\Z_r-\int_{P}Y_r\,d\Z_r\right)\right|\\
&\ \le\left|\int_{P'}Y_r\,d\Z_r-\int_{\widetilde{P}}Y_r\,d\Z_r\right|+\left|\int_{P'}Y_r\,d\Z_r-\int_{P}Y_r\,d\Z_r\right|\\
&\ \le C_{\al, T}(|Y'_0|+\|\Y\|_{\X; \al})(|Z'_0|+\|\Z\|_{\X; \al})(1+\|\X\|_{\al})^2\big(|\widetilde{P}|^{3\al-1}+|P|^{3\al-1}\big)|t-s|.
\end{align*}
The left-hand side can be made arbitrarily small when the mesh sizes of $P$ and $\widetilde{P}$ are small enough. 
The existence of $\lim_{|P|\to 0 }\int_{P}Y_r\,d\Z_r $ thus follows from the Cauchy criterion. Hence the definition of the integral $\int_s^t Y_r\,d\Z_r$ is well-defined.
\end{proof}

We derive an estimate for the controlled rough integral, which differs from the one obtained in~\cite[Theorem~1]{Gu04} and will be used in the next section.

\begin{coro}
With the setting in Lemma~\ref{thm:1},  there exists a constant $C_{\al, T}$ in $\RR$ such that the bound
\begin{equation}
\begin{aligned}
&\ \hspace{-3.5cm}\left|\int_s^t Y_r\,d\Z_r-\Big(Y_s  Z_{s,t}+Y'_sZ'_{s} \x_{s,t}\Big)\right|\\
\le&\ C_{\al, T}(|Y_0'|+\|\Y\|_{\X; \al})(|Z_0'|+\|\Z\|_{\X; \al})(1+\|\X\|_{\al})^2|t-s|^{3\al}
\end{aligned}
\mlabel{eq:21}
\end{equation}
holds for all $0\le s\leq t\le T$.
\mlabel{coro:esti}
\end{coro}

\begin{proof}
It is directly derived from~(\ref{eq:15}).
\end{proof}

\section{Solutions to Controlled RDEs}\label{ss:sec3}
The goal of this section is to study Controlled RDEs:
\[
dY = F(Y)\,d\mathbf Z, \quad Y_0=\xi,
\]
where $\Z=(Z,Z')$ is an $\X$-controlled rough path.
This differential equation is interpreted as an integral equation
\begin{equation}
Y_t=Y_0+\int_0^t F(Y_r)\,d\Z_r.
\mlabel{eq:add2}
\end{equation}
To this end, we first show that the controlled rough integral produces a controlled rough path. 
This extends the classical situation in which a controlled rough path is integrated against a rough path.

\begin{proposition}
Let $\alpha\in\left(\frac13,\frac12\right]$, $\X =(1, X, \x)\in \D$, $\Y = (Y, Y')\in \CC([0, T], \mathcal{L}(W, U))$ and $\Z = (Z, Z')\in \CC([0, T], W)$. Then
\[\Big(\int_0^{\bullet} Y_r\,d\Z_r, YZ' \Big)\in  \CC([0, T], U).\]
\mlabel{prop:1}
\end{proposition}

\begin{proof}
Set
\[
I_t:=\int_0^t Y_r\,d\Z_r,
\quad
I'_t:=Y_tZ'_t.
\]
We prove that $\mathbf I:=(I,I')\in C^\alpha_{\X}([0,T],U)$. By Definition~\ref{def:crp}, it suffices to prove that
\[
\|R^{0, \mathbf I}\|_{2\al}<\infty, \quad \|R^{1, \mathbf I}\|_{\al}<\infty,
\]
where
\begin{equation}
R^{0, \mathbf I}_{s, t}:=I_{s,t}-I'_sX_{s,t}=\int_s^t Y_r\,d\Z_r-Y_sZ'_sX_{s,t}, \quad R^{1, \mathbf I}_{s, t}:=I_t'-I_s'=Y_tZ'_t-Y_sZ'_s.
\mlabel{eq:26}
\end{equation}
First, $\|R^{0, \mathbf I}\|_{2\al}<\infty$ follows from
\begin{align*}
|R^{0, \mathbf I}_{s, t}|=&\ \Big| \int_s^t Y_r\,d\Z_r-Y_sZ'_sX_{s,t} \Big| \\
=&\ \Big| \Big(\int_s^t Y_r\,d\Z_r-Y_sZ_{s, t}-Y'_sZ'_s\x_{s, t}\Big)+ (Y_sZ_{s, t}+Y'_sZ'_s\x_{s, t}-Y_sZ'_sX_{s,t}) \Big|\\
\le&\  \Big|\int_s^t Y_r\,d\Z_r-Y_sZ_{s, t}-Y'_sZ'_s\x_{s, t}\Big|+ |Y_sZ_{s, t}-Y_sZ'_sX_{s, t}|+|Y'_sZ'_s\x_{s,t}|\\
\le&\  \Big|\int_s^t Y_r\,d\Z_r-Y_sZ_{s, t}-Y'_sZ'_s\x_{s, t}\Big|+ |Y_s|\,|Z_{s, t}-Z'_sX_{s,t}|+|Y'_s|\,|Z'_s|\,|\x_{s, t}|\\
=&\  \Big|\int_s^t Y_r\,d\Z_r-Y_sZ_{s, t}-Y'_sZ'_s\x_{s, t}\Big|+ |Y_s|\,|R^{0, \Z}_{s, t}|+|Y'_s|\,|Z'_s|\,|\x_{s, t}| \hspace{1cm} (\text{by (\ref{eq:-1})})\\
\le&\ C_{\al, T}(|Y_0'|+\|\Y\|_{\X; \al})(|Z_0'|+\|\Z\|_{\X; \al})(1+\|\X\|_{\al})^2|t-s|^{3\al}\\
&\ +|Y_s|\,\|R^{0, \Z}\|_{2\al}|t-s|^{2\al}+|Y'_s|\,|Z'_s|\,\|\x\|_{2\al}|t-s|^{2\al}\hspace{1cm} (\text{by (\ref{eq:21}) and (\ref{eq:remai})}).
\end{align*}
Further, $ \|R^{1, \mathbf I}\|_{\al}<\infty$ can be computed directly by
\begin{align*}
|R^{1, \mathbf I}_{s, t}| 
=|Y_tZ'_t-Y_sZ'_s|=&\  |(Y_t-Y_s)Z'_t+Y_s(Z'_t-Z'_s)|\\
\le&\ |Y_t-Y_s|\,|Z'_t|+|Y_s|\,|Z'_t-Z'_s|\\
\le&\  \|Y\|_{\al}|Z'_t||t-s|^{\al}+|Y_s|\,\|Z'\|_{\al}|t-s|^{\al}.\qedhere
\end{align*}
\end{proof}

We review that the composition of a controlled rough path with a regular function is still a controlled rough path. 

For $m\in \mathbb{Z}_{\geq 1}$, denote by $C^m(U ; W)$ the space of $m$ times continuously differentiable maps from Banach space $U$ to Banach space $W$.
For any $F\in  C^m(U ; W)$, define
\begin{equation*}
\|F\|_{\infty }:=\sup_{x\in U}\|F(x)\|.
\end{equation*}
A subspace $C^{m}_b(U; W) \subseteq  C^m(U ; W) $ is given by those $F$ in $C^m(U ; W) $ such that
\begin{equation*}
\|F\|_{C^{m}_b}:=\|F\|_{\infty }+\|DF\|_{\infty }+\cdots +\|D^mF\|_{\infty } < \infty.
\end{equation*}

\begin{lemma}\cite[Proposition 4]{Gu04}
Let $\alpha\in\left(\frac13,\frac12\right]$, $\X =(1, X, \x)\in \D$ and $\Y = (Y, Y')\in \CC([0, T], W)$. If $F\in C_b^2(W, U)$, then
\[F(\Y):=(F(Y), DF(Y)Y')\in \CC([0, T], U).\]
\mlabel{lem:2}
\end{lemma}

\begin{coro}
Let  $\alpha\in\left(\frac13,\frac12\right]$, $\X =(1, X, \x)\in \D$, $\Y = (Y, Y')\in \CC([0, T], U)$ and $\Z = (Z, Z')\in \CC([0, T], W)$. Assume that $F\in C_b^2(U, \mathcal{L}(W, U))$. Then
\begin{equation}
\left(\int_0^{\bullet} F(Y_r)\,d\Z_r, F(Y)Z' \right)\in  \CC([0, T], U).
\mlabel{eq:add1}
\end{equation}
\mlabel{coro:3}
\end{coro}

\begin{proof}
It follows from Proposition~\ref{prop:1} and Lemma~\ref{lem:2}.
\end{proof}

Before proving the following estimates, we record two elementary bounds that will be used repeatedly. 
We have
\begin{equation}
|Y'_s|\le |Y'_0|+|Y'_s-Y'_0|\le |Y'_0|+|s|^\al\|Y'\|_{\al}\le |Y'_0|+T^{\al}\|Y'\|_{\al}.
\mlabel{eq:26+3}
\end{equation}
Moreover, using the controlled expansion
\[
Y_{s,t}=Y'_sX_{s,t}+R^{0,\Y}_{s,t},
\]
we obtain
\begin{equation}
\begin{aligned}
\|Y\|_{\al}=&\ \sup_{0\le s\le t\le T}\frac{|Y_{s, t}|}{|t-s|^\al}\\
\le&\ \sup_{0\le s\le t\le T}\frac{|Y'_sX_{s, t}+R^{0,\Y}_{s, t}|}{|t-s|^\al} \\
\le&\ \sup_{0\le s\le t\le T}\frac{|Y'_sX_{s, t}|}{|t-s|^\al}+\sup_{0\le s\le t\le T}\frac{|R^{0,\Y}_{s, t}|}{|t-s|^\al}\\
\le&\ (|Y'_0|+T^{\al}\|Y'\|_{\al})\sup_{0\le s\le t\le T}\frac{|X_{s, t}|}{|t-s|^\al}+T^\al\sup_{0\le s\le t\le T}\frac{|R^{0,\Y}_{s, t}|}{|t-s|^{2\al}}\hspace{1cm} (\text{by (\ref{eq:26+3})}) \\
=&\ (|Y'_0|+T^{\al}\|Y'\|_{\al})\|X\|_{\al}+T^\al\|R^{0,\Y}\|_{2\al}. 
\end{aligned}
\mlabel{eq:26+7}
\end{equation}
Similarly,
\begin{equation}
|Y_s-\widetilde Y_s|\le |Y_0-\widetilde Y_0|+T^{\al}\Big((|Y'_0-\widetilde Y'_0|+T^{\al}\|Y'-\widetilde Y'\|_{\al})\|X\|_{\al}+T^\al\|R^{0,\Y-\widetilde \Y}\|_{2\al} \Big)
\mlabel{eq:27+5}
\end{equation}

To estimate the controlled rough path in~\meqref{eq:add1}, we first record two preliminary bounds in the next lemma.

\begin{lemma}
Let $\alpha\in(1/3,1/2]$ and $\X=(1,X,\mathbb X)\in \D$.

\begin{enumerate}
\item If $\Y=(Y,Y')\in C^\alpha_{\X}([0,T],W)$ and $F\in C_b^2(W,U)$, then
\[
F(\Y):=(F(Y),DF(Y)Y')\in C^\alpha_{\X}([0,T],U),
\]
and
\[
\|F(\Y)\|_{\X;\alpha}\le\|F\|_{C_b^2}M\bigl(T,|Y'_0|,\|\Y\|_{\X;\alpha},\|\X\|_\alpha\bigr).
\]

\item If $\Y,\widetilde{\Y}\in C^\alpha_{\X}([0,T],W)$ and $F\in C_b^3(W,U)$, then
\begin{equation}
\begin{aligned}
\|F(\Y)-F(\widetilde{\Y})\|_{\X;\alpha}
\le&\ \|F\|_{C_b^3}M\bigl(T,|Y'_0|,|\widetilde Y'_0|,\|\Y\|_{\X;\alpha},\|\widetilde{\Y}\|_{\X;\alpha},\|\X\|_\alpha\bigr) \\
&\ \times\left(\|\Y-\widetilde{\Y}\|_{\X;\alpha}+|Y_0-\widetilde Y_0|+|Y'_0-\widetilde Y'_0|\right).
\end{aligned}
\mlabel{eq:27+1}
\end{equation}
\end{enumerate}
\mlabel{lem:4}
\end{lemma}

\begin{proof}
Both assertions are standard consequences of the composition theorem for controlled rough paths and its quantitative local Lipschitz version; see~\cite[Theorem 4.1]{BG22}. 
The dependence of the constants displayed above follows from the elementary estimates~\eqref{eq:26+3},~\eqref{eq:26+7} and the analogous estimates for $\widetilde{\Y}$.
\end{proof}

To bound the right hand side of~\eqref{eq:add2}, we give the following estimates.

\begin{proposition}
With the setting in Corollary~\ref{coro:3}, let
\begin{equation}
\J:=\left(Y_0+\int_0^{\cdot} F(Y_r)\,d\Z_r, F(Y)Z' \right) \in  \CC([0, T], U).
\mlabel{eq:30-1}
\end{equation}
Then
\begin{align*}
\|\J\|_{\X; \al}\le&\  (1+C_{\al, T})\Big(T^{\al}M(T, \|F\|_{C_b^2}, \|\X\|_{\al}, |Y'_0|, |Z'_0|, \|\Y\|_{\X; \al}, \|\Z\|_{\X; \al})\nonumber\\
&\ \hspace{2cm}+\|F\|_{C_b^2}(1+|Y'_0|)(1+|Z'_0|)(1+\|\X\|_{\al})^2(1+\|\Z\|_{\X; \al})\Big),
\end{align*}
where $C_{\al, T}$ is a constant in $\RR$.
\end{proposition}

\begin{proof}
In view of~(\ref{eq:norm2}),
\[\|\J\|_{\X; \al}=\|R^{0,\J}\|_{2\al}+\|R^{1,\J}\|_{\al},\]
where
\begin{align*}
R^{0,\J}_{s,t}:=  \int_s^tF(Y_r)\,d\Z_r-F(Y_s)Z'_sX_{s, t},   \quad
R^{1,\J}_{s,t}:=  F(Y_t)Z'_t-F(Y_s)Z'_s.
\end{align*}
For the term $\|R^{0,\J}\|_{2\al}$,
\begin{align*}
\ |R^{0,\J}_{s,t}|
=&\ \Big|\Big(\int_s^tF(Y_r)\,d\Z_r-F(Y_s)Z_{s, t}-DF(Y_s)Y'_sZ'_s\x_{s, t} \Big)\\
&\ +\Big(F(Y_s)Z_{s, t}+DF(Y_s)Y'_sZ'_s\x_{s, t}-F(Y_s)Z'_sX_{s, t}\Big)\Big|\\
\le&\ \left|\int_s^tF(Y_r)\,d\Z_r-F(Y_s)Z_{s, t}-DF(Y_s)Y'_sZ'_s\x_{s, t} \right|\\
&\ +\Big|F(Y_s)Z_{s, t}+DF(Y_s)Y'_sZ'_s\x_{s, t}-F(Y_s)Z'_sX_{s, t}\Big|\\
\le&\ C_{\al, T}(|DF(Y_0)Y_0'|+\|F(\Y)\|_{\X; \al})(|Z_0'|+\|\Z\|_{\X; \al})(1+\|\X\|_{\al})^2|t-s|^{3\al}\\
&\ +\Big|F(Y_s)Z_{s, t}+DF(Y_s)Y'_sZ'_s\x_{s, t}-F(Y_s)Z'_sX_{s, t}\Big|\hspace{1cm} (\text{by (\ref{eq:21})})\\
=&\ C_{\al, T}(|DF(Y_0)Y_0'|+\|F(\Y)\|_{\X; \al})(|Z_0'|+\|\Z\|_{\X; \al})(1+\|\X\|_{\al})^2|t-s|^{3\al}\\
&\ +\Big|F(Y_s)(Z_s'X_{s, t}+R_{s, t}^{\Z, 0})+DF(Y_s)Y'_sZ'_s\x_{s, t}-F(Y_s)Z'_sX_{s, t}\Big|\hspace{1cm} (\text{by (\ref{eq:-1})})\\
=&\ C_{\al, T}(|DF(Y_0)Y_0'|+\|F(\Y)\|_{\X; \al})(|Z_0'|+\|\Z\|_{\X; \al})(1+\|\X\|_{\al})^2|t-s|^{3\al}\\
&\ +|F(Y_s)R_{s, t}^{\Z, 0}+DF(Y_s)Y'_sZ'_s\x_{s, t}|\\
\le&\ C_{\al, T}(\|DF\|_{\infty}|Y_0'|+\|F(\Y)\|_{\X; \al})(|Z_0'|+\|\Z\|_{\X; \al})(1+\|\X\|_{\al})^2|t-s|^{3\al}\\
&\ +\|F\|_{\infty}\|R^{\Z, 0}\|_{2\al}|t-s|^{2\al}\\
&\ +\|DF\|_{\infty}(|Y'_0|+T^{\al}\|Y'\|_{\al})(|Z'_0|+T^{\al}\|Z'\|_{\al})\|\x\|_{2\al}|t-s|^{2\al}\hspace{1cm} (\text{by (\ref{eq:26+3})}).
\end{align*}
Consequently,
\begin{align*}
\|R^{0,\J}\|_{2\al}\le&\  C_{\al, T}T^\al(\|DF\|_{\infty}|Y_0'|+\|F(\Y)\|_{\X; \al})(|Z_0'|+\|\Z\|_{\X; \al})(1+\|\X\|_{\al})^2\\
&\ +\|F\|_{\infty}\|R^{\Z, 0}\|_{2\al}+\|DF\|_{\infty}(|Y'_0|+T^{\al}\|Y'\|_{\al})(|Z'_0|+T^{\al}\|Z'\|_{\al})\|\x\|_{2\al}.
\end{align*}

For the term $\|R^{1,\J}\|_{\al}$,
\begin{align*}
|R^{1,\J}_{s, t}|=&\ |F(Y_t)Z'_t-F(Y_s)Z'_s| \\
=&\ \big|\big(F(Y_t)-F(Y_s)\big)Z'_t+F(Y_s)\big(Z'_t-Z'_s\big)\big|\\
\le&\ |F(Y_t)-F(Y_s)|\,|Z'_t|+|F(Y_s)|\,|Z'_t-Z'_s|\\
\le&\ \|DF\|_{\infty}|Y_t-Y_s|\,|Z'_t|+\|F\|_{\infty}|Z'_t-Z'_s| \\
\le&\ \|DF\|_{\infty}\|Y\|_{\al}(|Z'_0|+T^{\al}\|Z'\|_{\al})|t-s|^\al+\|F\|_{\infty}\|Z'\|_{\al}|t-s|^\al \hspace{1cm} (\text{by (\ref{eq:26+3})})\\
\le&\ \|DF\|_{\infty}\Big((|Y'_0|+T^{\al}\|Y'\|_{\al})\|X\|_{\al}+T^\al\|R^{0,\Y}\|_{2\al}\Big)\nonumber\\
&\ \times(|Z'_0|+T^{\al}\|Z'\|_{\al})|t-s|^\al+\|F\|_{\infty}\|Z'\|_{\al}|t-s|^\al \hspace{1cm} (\text{by (\ref{eq:26+7})}).
\end{align*}
Thus
\begin{equation*}
\|R^{1,\J}\|_{\al}\le \|DF\|_{\infty}\Big((|Y'_0|+T^{\al}\|Y'\|_{\al})\|X\|_{\al}+T^\al\|R^{0,\Y}\|_{2\al}\Big)(|Z'_0|+T^{\al}\|Z'\|_{\al})+\|F\|_{\infty}\|Z'\|_{\al}.
\end{equation*}
Summing up the above estimates of the two steps, we conclude
\begin{align*}
\|\J\|_{\X; \al}
\le&\ C_{\al, T}T^\al(\|DF\|_{\infty}|Y_0'|+\|F(\Y)\|_{\X; \al})(|Z_0'|+\|\Z\|_{\X; \al})(1+\|\X\|_{\al})^2\\
&\ +\|DF\|_{\infty}(|Y'_0|+T^{\al}\|Y'\|_{\al})(|Z'_0|+T^{\al}\|Z'\|_{\al})\|\x\|_{2\al}\\
&\ +\|DF\|_{\infty}\Big((|Y'_0|+T^{\al}\|Y'\|_{\al})\|X\|_{\al}+T^\al\|R^{0,\Y}\|_{2\al}\Big)(|Z'_0|+T^{\al}\|Z'\|_{\al})\\
&\ +\|F\|_{\infty}\|Z'\|_{\al} \\
\le&\ (1+C_{\al, T})\bigg(T^{\al}\Big((\|DF\|_{\infty}|Y_0'|+\|F(\Y)\|_{\X; \al})(|Z_0'|+\|\Z\|_{\X; \al})(1+\|\X\|_{\al})^2\\
&\ +\|DF\|_{\infty}\|Y'\|_{\al}(|Z'_0|+T^{\al}\|Z'\|_{\al})\|\x\|_{2\al}+\|DF\|_{\infty}|Y'_0|\,\|Z'\|_{\al}\,\|\x\|_{2\al}\\
&\ +\|DF\|_{\infty}|Y'_0|\,\|X\|_{\al}\|Z'\|_{\al}+\|DF\|_{\infty}\|Y'\|_{\al}\|X\|_{\al}(|Z'_0|+T^{\al}\|Z'\|_{\al})\\
&\ +\|DF\|_{\infty}\|R^{0,\Y}\|_{2\al}(|Z'_0|+T^{\al}\|Z'\|_{\al})\Big)\\
&\ +\Big(\|DF\|_{\infty}|Y'_0|\,|Z'_0|\,\|\x\|_{2\al}+ \|DF\|_{\infty}|Y'_0|\,\|X\|_{\al}|Z'_0|+\|F\|_{\infty}\|Z'\|_{\al} \Big)   \bigg) \\
\le&\ (1+C_{\al, T})\Big(T^{\al}M(T, \|F\|_{C_b^2}, \|\X\|_{\al}, |Y'_0|, |Z'_0|, \|\Y\|_{\X; \al}, \|\Z\|_{\X; \al})\\
&\ +\|DF\|_{\infty}|Y'_0|\,|Z'_0|\,\|\x\|_{2\al}+ \|DF\|_{\infty}|Y'_0|\,\|X\|_{\al}|Z'_0|+\|F\|_{\infty}\|Z'\|_{\al} \Big)(\text{by Lemma~\ref{lem:4}})\\
\le&\ (1+C_{\al, T})\Big(T^{\al}M(T, \|F\|_{C_b^2}, \|\X\|_{\al}, |Y'_0|, |Z'_0|, \|\Y\|_{\X; \al}, \|\Z\|_{\X; \al})\\
&\ +\|F\|_{C_b^2}(1+|Y'_0|)(1+|Z'_0|)(1+\|\X\|_{\al})^2(1+\|\Z\|_{\X; \al})\Big).\qedhere
\end{align*}
\end{proof}

We establish the stability of the rough integral $\J$.

\begin{proposition}
Let $\alpha\in\left(\frac13,\frac12\right]$, $\X =(1, X, \x)\in \D$, $\Y = (Y, Y')\in \CC([0, T], W)$ and $\widetilde{\Y}  = (\widetilde Y, \widetilde Y')\in \CC([0, T], W)$.
Let $\Z = (Z, Z')\in \CC([0, T], U)$ and $F\in C_b^3(W, \mathcal{L}(U, W))$. Then
\begin{equation}
\begin{aligned}
d_{\X, \X; \al}(\J, \widetilde \J)
\le&\  C_{\al, T}M(T, \|F\|_{C_b^3}, |Y'_0|, |\widetilde Y'_0|, |Z'_0|, \|\Y\|_{\X; \al}, \|\widetilde \Y\|_{\X; \al}, \|\X\|_{\al}, \|\Z\|_{\X; \al})\\
&\ \times\Big(T^{\al}\|\Y-\widetilde \Y\|_{\X; \al}+|Y_0-\widetilde Y_0|+|Y'_0-\widetilde Y'_0|\Big),
\end{aligned}
\mlabel{eq:djj}
\end{equation}
where
\[\widetilde \J:=\left(\widetilde Y_0+\int_0^{\bullet} F(\widetilde Y_r)\,d\Z_r, F(\widetilde Y)Z' \right).\]
\mlabel{prop:3}
\end{proposition}

\begin{proof}
It follows from~(\ref{eq:norm1}) that
\[d_{\X, \X; \al}(\J, \widetilde \J)=\|R^{0,\J}-R^{0, \widetilde \J}\|_{2\al}+\|R^{1,\J}-R^{1, \widetilde \J}\|_{\al},\]
where

\begin{align*}
R^{0,\J}_{s,t}-R^{0,\widetilde \J}_{s,t}=&\  \left( \int_s^tF(Y_r)\,d\Z_r-F(Y_s)Z'_sX_{s, t}\right) -\left( \int_s^tF(\widetilde Y_r)\,d\Z_r-F(\widetilde Y_s)Z'_sX_{s, t}\right) ,  \\
R^{1,\J}_{s,t}-R^{1,\widetilde \J}_{s,t}=&\ \big(F(Y_t)Z'_t-F(Y_s)Z'_s\big)-\big(F(\widetilde Y_t)Z'_t-F(\widetilde Y_s)Z'_s\big).
\end{align*}
The remainder of the proof is divided into two steps.

\noindent{\bf Step 1. } For the term $\|R^{0,\J}-R^{0, \widetilde \J}\|_{2\al}$,
\begin{align}
&\ \hspace{-1.6cm}|R^{0,\J}_{s,t}-R^{0,\widetilde \J}_{s,t}|\nonumber\\
=&\ \bigg|\int_s^tF(Y_r)-F(\widetilde Y_r)\,d\Z_r - \big(F(Y_s)-F(\widetilde Y_s)\big)Z'_sX_{s, t}\bigg|\nonumber\\
=&\ \bigg|\bigg(\int_s^tF(Y_r)-F(\widetilde Y_r)\,d\Z_r-\big(F(Y_s)-F(\widetilde Y_s)\big)Z_{s, t}\nonumber\\
&\ \hspace{0.5cm}-\big(DF(Y_s)Y'_s-DF(\widetilde Y_s)\widetilde Y'_s\big)Z'_s\x_{s, t} \bigg) \nonumber\\
&\ +\bigg(\big(F(Y_s)-F(\widetilde Y_s)\big)Z_{s, t}+\big(DF(Y_s)Y'_s-DF(\widetilde Y_s)\widetilde Y'_s\big)Z'_s\x_{s, t}\nonumber\\
&\  \ \hspace{0.5cm}-\big(F(Y_s)-F(\widetilde Y_s)\big)Z'_sX_{s, t}\bigg)\bigg|\nonumber\\
\le&\ \bigg|\int_s^tF(Y_r)-F(\widetilde Y_r)\,d\Z_r-\big(F(Y_s)\nonumber\\
&\ -F(\widetilde Y_s)\big)Z_{s, t}-\big(DF(Y_s)Y'_s-DF(\widetilde Y_s)\widetilde Y'_s\big)Z'_s\x_{s, t} \bigg| \nonumber\\
&\ +\Big|\big(F(Y_s)-F(\widetilde Y_s)\big)Z_{s, t}+\big(DF(Y_s)Y'_s-DF(\widetilde Y_s)\widetilde Y'_s\big)Z'_s\x_{s, t}\nonumber\\
&\ \hspace{0.5cm}-\big(F(Y_s)-F(\widetilde Y_s)\big)Z'_sX_{s, t}\Big|\nonumber\\
\le&\ C_{\al, T}\big(|DF(Y_0)Y_0'-DF(\widetilde Y_0)\widetilde Y_0'|+\|F(\Y)-F(\widetilde \Y)\|_{\X; \al}\big)\nonumber\\
&\ \times(|Z_0'|+\|\Z\|_{\X; \al})(1+\|\X\|_{\al})^2|t-s|^{3\al}\nonumber\\
&\ +\Big|\big(F(Y_s)-F(\widetilde Y_s)\big)Z_{s, t}+\big(DF(Y_s)Y'_s-DF(\widetilde Y_s)\widetilde Y'_s\big)Z'_s\x_{s, t}\mlabel{eq:37}\\
&\ \hspace{0.5cm}-\big(F(Y_s)-F(\widetilde Y_s)\big)Z'_sX_{s, t}\Big|\hspace{1cm} (\text{by (\ref{eq:21})})\nonumber\\
=&\ C_{\al, T}\big(|DF(Y_0)Y_0'-DF(\widetilde Y_0)\widetilde Y_0'|+\|F(\Y)-F(\widetilde \Y)\|_{\X; \al}\big)\nonumber\\
&\ \times(|Z_0'|+\|\Z\|_{\X; \al})(1+\|\X\|_{\al})^2|t-s|^{3\al}\nonumber\\
&\ +\Big|\big(F(Y_s)-F(\widetilde Y_s)\big)(Z'_sX_{s, t}+R^{0,\Z}_{s,t}) +\big(DF(Y_s)Y'_s-DF(\widetilde Y_s)\widetilde Y'_s\big)Z'_s\x_{s, t}\nonumber\\
&\hspace{0.5cm}- \big(F(Y_s)-F(\widetilde Y_s)\big)Z'_sX_{s, t}\Big| \hspace{1cm} (\text{by (\ref{eq:-1})})\nonumber\\
=&\ C_{\al, T}\big(|DF(Y_0)Y_0'-DF(\widetilde Y_0)\widetilde Y_0'|+\|F(\Y)-F(\widetilde \Y)\|_{\X; \al}\big)\nonumber\\
&\ \times(|Z_0'|+\|\Z\|_{\X; \al})(1+\|\X\|_{\al})^2|t-s|^{3\al}\nonumber\\
&\ +\Big|\big(F(Y_s)-F(\widetilde Y_s)\big)R^{0,\Z}_{s,t}+\big(DF(Y_s)Y'_s-DF(\widetilde Y_s)\widetilde Y'_s\big)Z'_s\x_{s, t}\Big|\nonumber\\
\le&\ C_{\al, T}\big(|DF(Y_0)Y_0'-DF(\widetilde Y_0)\widetilde Y_0'|+\|F(\Y)-F(\widetilde \Y)\|_{\X; \al}\big)\nonumber\\
&\ \times(|Z_0'|+\|\Z\|_{\X; \al})(1+\|\X\|_{\al})^2|t-s|^{3\al}+|F(Y_s)-F(\widetilde Y_s)|\,\|R^{0,\Z}\|_{2\al}|t-s|^{2\al}\nonumber\\
&\ +|DF(Y_s)Y'_s-DF(\widetilde Y_s)\widetilde Y'_s|\,|Z'_s|\,\|\x\|_{2\al}|t-s|^{2\al}\nonumber\\
\le&\ C_{\al, T}T^{\al}\big(|DF(Y_0)Y_0'-DF(\widetilde Y_0)\widetilde Y_0'|+\|F(\Y)-F(\widetilde \Y)\|_{\X; \al}\big)\nonumber\\
&\ \times(|Z_0'|+\|\Z\|_{\X; \al})(1+\|\X\|_{\al})^2|t-s|^{2\al}+|F(Y_s)-F(\widetilde Y_s)|\,\|R^{0,\Z}\|_{2\al}|t-s|^{2\al}\nonumber\\
&\ +|DF(Y_s)Y'_s-DF(\widetilde Y_s)\widetilde Y'_s|\big(|Z'_0|+T^{\al}\|Z'\|_{\al} \big)\|\x\|_{2\al}|t-s|^{2\al}\hspace{1cm} (\text{by (\ref{eq:26+3})}). \nonumber
\end{align}
We are going to estimate the terms $|F(Y_s)-F(\widetilde Y_s)|$ and $|DF(Y_s)Y'_s-DF(\widetilde Y_s)\widetilde Y'_s|$ in~\eqref{eq:37} separately.

\noindent{\bf Step 1.1. } For $|F(Y_s)-F(\widetilde Y_s)|$,
\begin{equation}
\begin{aligned}
&\ \hspace{-1cm}|F(Y_s)-F(\widetilde Y_s)|\\
\le&\ \|DF\|_{\infty}|Y_s-\widetilde Y_s|\\
\le&\ \|DF\|_{\infty}\Big(|Y_0-\widetilde Y_0|\\
&\ +T^{\al}\big((|Y'_0-\widetilde Y'_0|+T^{\al}\|Y'-\widetilde Y'\|_{\al})\|X\|_{\al}+T^\al\|R^{0,\Y-\widetilde \Y}\|_{2\al} \big) \Big) \hspace{1cm} (\text{by (\ref{eq:27+5})}).\\
=&\ \|DF\|_{\infty}\Big(|Y_0-\widetilde Y_0|+T^{\al}\|X\|_{\al}|Y'_0-\widetilde Y'_0| \Big)\\
&\ +\|DF\|_{\infty}T^{\al}\Big(T^{\al}\|Y'-\widetilde Y'\|_{\al}\|X\|_{\al}+\|R^{0,\Y-\widetilde \Y}\|_{2\al} \Big)\\
\le&\ \|DF\|_{\infty}\Big(|Y_0-\widetilde Y_0|+T^{\al}\|X\|_{\al}|Y'_0-\widetilde Y'_0| \Big)\\
&\ +\|DF\|_{\infty}T^{\al}(1+T^{\al})(1+\|X\|_{\al})\|\Y-\widetilde \Y\|_{\X; \al}. 
\end{aligned}
\mlabel{eq:39}
\end{equation}

\noindent{\bf Step 1.2. } For $|DF(Y_s)Y'_s-DF(\widetilde Y_s)\widetilde Y'_s|$,
\begin{align}
&\ \hspace{-1.7cm}|DF(Y_s)Y'_s-DF(\widetilde Y_s)\widetilde Y'_s|\nonumber\\
=&\ \big|\big(DF(Y_s)-DF(\widetilde Y_s)\big)Y'_s+DF(\widetilde Y_s)(Y'_s-\widetilde Y'_s)\big|\nonumber\\
\le&\ |DF(Y_s)-DF(\widetilde Y_s)|\,|Y'_s|+|DF(\widetilde Y_s)|\,|Y'_s-\widetilde Y'_s|\nonumber\\
\le&\ \|D^2F\|_{\infty}|Y_s-\widetilde Y_s|\,|Y'_s|+\|DF\|_{\infty}|Y'_s-\widetilde Y'_s|   \nonumber\\
\le&\ \|D^2F\|_{\infty}\Big(|Y_0-\widetilde Y_0|+T^{\al} \nonumber\\
&\ \times\big((|Y'_0-\widetilde Y'_0|+T^{\al}\|Y'-\widetilde Y'\|_{\al})\|X\|_{\al}+T^\al\|R^{0,\Y-\widetilde \Y}\|_{2\al} \big) \Big)(|Y'_0|+T^{\al}\|Y'\|_{\al} ) \nonumber\\
&\ +\|DF\|_{\infty}(|Y'_0-\widetilde Y'_0|+T^{\al}\|Y'-\widetilde Y'\|_{\al} ) \hspace{1cm} (\text{by (\ref{eq:26+3}) and (\ref{eq:27+5})})  \nonumber\\
=&\ \|F\|_{C_b^2}\Big((|Y'_0|+T^{\al}\|Y'\|_{\al} )|Y_0-\widetilde Y_0|\mlabel{eq:40}\\
&\ \hspace{1.3cm}+(|Y'_0|+T^{\al}\|Y'\|_{\al} )T^{\al}\|X\|_{\al}|Y'_0-\widetilde Y'_0|+|Y'_0-\widetilde Y'_0|\Big)\nonumber\\
&\ +\|F\|_{C_b^2}T^{\al}\Big((|Y'_0|+T^{\al}\|Y'\|_{\al})\|X\|_{\al}T^{\al}\|Y'-\widetilde Y'\|_{\al}\nonumber\\
&\ \hspace{2.3cm}+(|Y'_0|+T^{\al}\|Y'\|_{\al} )\|R^{0,\Y-\widetilde \Y}\|_{2\al}+\|Y'-\widetilde Y'\|_{\al} \Big) \nonumber\\
\le&\ \|F\|_{C_b^2}\Big((|Y'_0|+T^{\al}\|Y'\|_{\al} )|Y_0-\widetilde Y_0|\nonumber\\
&\ \hspace{1.3cm}+(|Y'_0|+T^{\al}\|Y'\|_{\al} )T^{\al}\|X\|_{\al}|Y'_0-\widetilde Y'_0|+|Y'_0-\widetilde Y'_0|\Big)\nonumber\\
&\ +\|F\|_{C_b^2}T^{\al}(1+|Y'_0|+T^{\al}\|Y'\|_{\al})(1+T^{\al})(1+\|X\|_{\al})\|\Y-\widetilde \Y\|_{\X; \al}. \nonumber
\end{align}
Substituting~(\ref{eq:27+1}),~(\ref{eq:39}) and~(\ref{eq:40}) into~(\ref{eq:37}),
\begin{align*}
\frac{|R^{0,\J}_{s,t}-R^{0,\widetilde \J}_{s,t}|}{|t-s|^{2\al}}
\le&\  C_{\al, T}M(T, \|F\|_{C_b^3}, |Y'_0|, |\widetilde Y'_0|, |Z'_0|, \|\Y\|_{\X; \al}, \|\widetilde \Y\|_{\X; \al}, \|\X\|_{\al}, \|\Z\|_{\X; \al})\\
&\ \times\Big(T^{\al}\|\Y-\widetilde \Y\|_{\X; \al}+|Y_0-\widetilde Y_0|+|Y'_0-\widetilde Y'_0|\Big).
\end{align*}
Consequently,
\begin{equation}
\begin{aligned}
\|R^{0,\J}-R^{0, \widetilde \J}\|_{2\al}
\le&\  C_{\al, T}M(T, \|F\|_{C_b^3}, |Y'_0|, |\widetilde Y'_0|, |Z'_0|, \|\Y\|_{\X; \al}, \|\widetilde \Y\|_{\X; \al}, \|\X\|_{\al}, \|\Z\|_{\X; \al})\\
&\ \times\Big(T^{\al}\|\Y-\widetilde \Y\|_{\X; \al}+|Y_0-\widetilde Y_0|+|Y'_0-\widetilde Y'_0|\Big).
\end{aligned}
\mlabel{eq:41}
\end{equation}

\noindent{\bf Step 2. } For the term $\|R^{1,\J}-R^{1, \widetilde \J}\|_{\al}$,
\begin{align}
|R^{1,\J}_{s, t}-R^{1, \widetilde \J}_{s, t}|
=&\ \Big|\Big(\big(F(Y_t)-F(Y_s)\big)Z'_t+F(Y_s)(Z'_t-Z'_s)\Big)\nonumber\\
&\ -\Big(\big(F(\widetilde Y_t)-F(\widetilde Y_s)\big)Z'_t+F(\widetilde Y_s)(Z'_t-Z'_s)\Big)\Big|\mlabel{eq:42}\\
\le&\ \Big|\Big(F(Y_t)-F(Y_s)\Big)-\Big(F(\widetilde Y_t)-F(\widetilde Y_s)\Big)\Big|\,|Z'_t|+|F(Y_s)-F(\widetilde Y_s)|\,|Z'_t-Z'_s|. \nonumber
\end{align}
For convenience, denote
\begin{align*}
I_5:=  \Big|\Big(F(Y_t)-F(Y_s)\Big)-\Big(F(\widetilde Y_t)-F(\widetilde Y_s)\Big)\Big|\,|Z'_t|, \quad
I_6:=  |F(Y_s)-F(\widetilde Y_s)|\,|Z'_t-Z'_s|.
\end{align*}

\noindent{\bf Step 2.1. }  The estimate of $I_5$ is as follows
\begin{align}
I_5
=&\ \Big|\int_0^1DF(Y_s+\theta Y_{s, t})Y_{s, t}d\theta -\int_0^1DF(\widetilde Y_s+\theta \widetilde Y_{s, t})\widetilde Y_{s, t}d\theta  \Big|\,|Z'_t| \nonumber\\
=&\ \Big|\int_0^1\Big(DF(Y_s+\theta Y_{s, t})Y_{s, t}-DF(\widetilde Y_s+\theta \widetilde Y_{s, t})\widetilde Y_{s, t}\Big)d\theta\Big|\,|Z'_t|\mlabel{eq:43}\\
\le&\ (|Z'_0|+T^{\al}\|Z'\|_{\al})\Big|\int_0^1\Big(DF(Y_s+\theta Y_{s, t})Y_{s, t}-DF(\widetilde Y_s+\theta \widetilde Y_{s, t})\widetilde Y_{s, t}\Big)d\theta\Big|\hspace{1cm} (\text{by (\ref{eq:26+3})}). \nonumber
\end{align}
For the second multiplicative factor,
\begin{align}
&\ \hspace{-1cm}\Big|\int_0^1\Big(DF(Y_s+\theta Y_{s, t})Y_{s, t}-DF(\widetilde Y_s+\theta \widetilde Y_{s, t})\widetilde Y_{s, t}\Big)d\theta\Big|\nonumber\\
=&\ \bigg|\int_0^1\bigg(\Big(DF(Y_s+\theta Y_{s, t})-DF(\widetilde Y_s+\theta \widetilde Y_{s, t})\Big)Y_{s, t}+DF(\widetilde Y_s+\theta \widetilde Y_{s, t})\Big(Y_{s, t}-\widetilde Y_{s, t}\Big)\bigg)d\theta\bigg|\nonumber\\
\le&\ \Big|\int_0^1\Big(DF(Y_s+\theta Y_{s, t})-DF(\widetilde Y_s+\theta \widetilde Y_{s, t})\Big)Y_{s, t}d\theta\Big|\nonumber\\
&\ +\Big|\int_0^1DF(\widetilde Y_s+\theta \widetilde Y_{s, t})\Big(Y_{s, t}-\widetilde Y_{s, t}\Big)d\theta\Big|\nonumber\\
\le&\ \int_0^1\|D^2F\|_{\infty}|Y_s-\widetilde Y_s+\theta(Y_{s, t}-\widetilde Y_{s, t})|\,|Y_{s, t}|d\theta+\int_0^1\|DF\|_{\infty}|Y_{s, t}-\widetilde Y_{s, t}|d\theta\nonumber\\
\le&\ \int_0^1\|D^2F\|_{\infty}|Y_s-\widetilde Y_s|\,|Y_{s, t}|d\theta+\int_0^1\|D^2F\|_{\infty}\theta|Y_{s, t}-\widetilde Y_{s, t}|\,|Y_{s, t}|d\theta\nonumber\\
&\ +\int_0^1\|DF\|_{\infty}|Y_{s, t}-\widetilde Y_{s, t}|d\theta\nonumber\\
=&\ \|D^2F\|_{\infty}|Y_s-\widetilde Y_s|\,|Y_{s, t}|+\frac{1}{2}\|D^2F\|_{\infty}|Y_{s, t}-\widetilde Y_{s, t}|\,|Y_{s, t}|+\|DF\|_{\infty}|Y_{s, t}-\widetilde Y_{s, t}|\nonumber\\
\le&\ \|D^2F\|_{\infty}|Y_s-\widetilde Y_s|\,\|Y\|_{\al}|t-s|^{\al}+\frac{1}{2}T^{\al}\|D^2F\|_{\infty}\|Y-\widetilde Y\|_{\al}\|Y\|_{\al}|t-s|^{\al}\nonumber\\
&\ +\|DF\|_{\infty}\|Y-\widetilde Y\|_{\al}|t-s|^{\al}\mlabel{eq:44}\\
\le&\ \|D^2F\|_{\infty}\bigg(|Y_0-\widetilde Y_0|+T^{\al}\Big((|Y'_0-\widetilde Y'_0|+T^{\al}\|Y'-\widetilde Y'\|_{\al})\|X\|_{\al}+T^\al\|R^{0,\Y-\widetilde \Y}\|_{2\al} \Big)\bigg)\nonumber\\
&\ \times\Big((|Y'_0|+T^{\al}\|Y'\|_{\al})\|X\|_{\al}+T^\al\|R^{0,\Y}\|_{2\al}\Big)|t-s|^{\al}\nonumber\\
&\ +\frac{1}{2}T^{\al}\|D^2F\|_{\infty}\Big((|Y'_0-\widetilde Y'_0|+T^{\al}\|Y'-\widetilde Y'\|_{\al})\|X\|_{\al}+T^\al\|R^{0,\Y-\widetilde \Y}\|_{2\al} \Big)  \nonumber\\
&\ \ \ \ \times\Big((|Y'_0|+T^{\al}\|Y'\|_{\al})\|X\|_{\al}+T^\al\|R^{0,\Y}\|_{2\al}\Big)|t-s|^{\al}\nonumber\\
&\ +\|DF\|_{\infty}\Big((|Y'_0-\widetilde Y'_0|+T^{\al}\|Y'-\widetilde Y'\|_{\al})\|X\|_{\al}+T^\al\|R^{0,\Y-\widetilde \Y}\|_{2\al} \Big)|t-s|^{\al}    \nonumber\\
&\ \hspace{10cm} (\text{by (\ref{eq:26+7}) and (\ref{eq:27+5})})\nonumber\\
\le&\ \|D^2F\|_{\infty}\bigg(|Y_0-\widetilde Y_0|+T^{\al}\Big((|Y'_0-\widetilde Y'_0|+T^{\al}\|\Y-\widetilde \Y\|_{\X; \al})\|\X\|_{\al}+T^\al\|\Y-\widetilde \Y\|_{\X; \al} \Big)\bigg)\nonumber\\
&\ \times\Big((|Y'_0|+T^{\al}\|\Y\|_{\X; \al})\|\X\|_{\al}+T^\al\|\Y\|_{\X; \al}\Big)|t-s|^{\al}\nonumber\\
&\ +\frac{1}{2}T^{\al}\|D^2F\|_{\infty}\Big((|Y'_0-\widetilde Y'_0|+T^{\al}\|\Y-\widetilde \Y\|_{\X; \al})\|\X\|_{\al}+T^\al\|\Y-\widetilde \Y\|_{\X; \al} \Big)  \nonumber\\
&\ \ \ \ \times\Big((|Y'_0|+T^{\al}\|\Y\|_{\X; \al})\|\X\|_{\al}+T^\al\|\Y\|_{\X; \al}\Big)|t-s|^{\al}\nonumber\\
&\ +\|DF\|_{\infty}\Big((|Y'_0-\widetilde Y'_0|+T^{\al}\|\Y-\widetilde \Y\|_{\X; \al})\|\X\|_{\al}+T^\al\|\Y-\widetilde \Y\|_{\X; \al} \Big)|t-s|^{\al}. \nonumber
\end{align}
Substituting~(\ref{eq:44}) into~(\ref{eq:43}),
\begin{align}
&\ \hspace{-1cm}\frac{I_5}{|t-s|^\al}\nonumber\\
\le&\ (|Z'_0|+T^{\al}\|\Z\|_{\X; \al})\Bigg(\|D^2F\|_{\infty}\bigg(|Y_0-\widetilde Y_0|+T^{\al}\Big((|Y'_0-\widetilde Y'_0|+T^{\al}\|\Y-\widetilde \Y\|_{\X; \al})\|\X\|_{\al}\nonumber\\
&\ +T^\al\|\Y-\widetilde \Y\|_{\X; \al} \Big)\bigg)\times\Big((|Y'_0|+T^{\al}\|\Y\|_{\X; \al})\|\X\|_{\al}+T^\al\|\Y\|_{\X; \al}\Big)\nonumber\\
&\ +\frac{1}{2}T^{\al}\|D^2F\|_{\infty}\Big((|Y'_0-\widetilde Y'_0|+T^{\al}\|\Y-\widetilde \Y\|_{\X; \al})\|\X\|_{\al}+T^\al\|\Y-\widetilde \Y\|_{\X; \al} \Big)  \mlabel{eq:45}\\
&\ \ \ \ \times\Big((|Y'_0|+T^{\al}\|\Y\|_{\X; \al})\|\X\|_{\al}+T^\al\|\Y\|_{\X; \al}\Big)\nonumber\\
&\ +\|DF\|_{\infty}\Big((|Y'_0-\widetilde Y'_0|+T^{\al}\|\Y-\widetilde \Y\|_{\X; \al})\|\X\|_{\al}+T^\al\|\Y-\widetilde \Y\|_{\X; \al} \Big)\Bigg). \nonumber
\end{align}

\noindent{\bf Step 2.2. } Turning to $I_6$,
\begin{align*}
I_6\le&\ |F(Y_s)-F(\widetilde Y_s)|\,\|Z'\|_{\al}|t-s|^{\al}\nonumber\\
\le&\ \|DF\|_{\infty}|Y_s-\widetilde Y_s|\,\|Z'\|_{\al}|t-s|^{\al}    \nonumber\\
\le&\ \|DF\|_{\infty}\bigg(|Y_0-\widetilde Y_0|+T^{\al}\Big((|Y'_0-\widetilde Y'_0|+T^{\al}\|Y'-\widetilde Y'\|_{\al})\|X\|_{\al}+T^\al\|R^{0,\Y-\widetilde \Y}\|_{2\al} \Big)\bigg)\\
&\ \times\|\Z\|_{\X; \al}|t-s|^{\al} \hspace{1cm} (\text{by (\ref{eq:27+5})}).
\end{align*}
Then
\begin{equation}
\begin{aligned}
&\ \hspace{-0.7cm}\frac{I_6}{|t-s|^{\al}}\\
\le&\  \|DF\|_{\infty}\bigg(|Y_0-\widetilde Y_0|+T^{\al}\Big((|Y'_0-\widetilde Y'_0|+T^{\al}\|Y'-\widetilde Y'\|_{\al})\|X\|_{\al}+T^\al\|R^{0,\Y-\widetilde \Y}\|_{2\al} \Big)\bigg)\\
&\ \times\|\Z\|_{\X; \al}.
\end{aligned}
\mlabel{eq:46}
\end{equation}
Substituting~(\ref{eq:45}) and~(\ref{eq:46}) into~(\ref{eq:42}),
\begin{align*}
\frac{|R^{1,\J}_{s,t}-R^{1,\widetilde \J}_{s,t}|}{|t-s|^{\al}}
\le&\  C_{\al, T}M(T, \|F\|_{C_b^3}, |Y'_0|, |\widetilde Y'_0|,|Z'_0|, \|\Y\|_{\X; \al}, \|\widetilde \Y\|_{\X; \al}, \|\X\|_{\al}, \|\Z\|_{\X; \al})\\
&\ \times\Big(T^{\al}\|\Y-\widetilde \Y\|_{\X; \al}+|Y_0-\widetilde Y_0|+|Y'_0-\widetilde Y'_0|\Big).
\end{align*}
Hence
\begin{equation}
\begin{aligned}
\|R^{1,\J}-R^{1, \widetilde \J}\|_{\al}
\le&\  C_{\al, T}M(T, \|F\|_{C_b^3}, |Y'_0|, |Z'_0|, |\widetilde Y'_0|, \|\Y\|_{\X; \al}, \|\widetilde \Y\|_{\X; \al}, \|\X\|_{\al}, \|\Z\|_{\X; \al})\\
&\ \times\Big(T^{\al}\|\Y-\widetilde \Y\|_{\X; \al}+|Y_0-\widetilde Y_0|+|Y'_0-\widetilde Y'_0|\Big).
\end{aligned}
\mlabel{eq:47}
\end{equation}

Finally, combining~(\ref{eq:41}) and~(\ref{eq:47}) yields the desired~(\mref{eq:djj}).
\end{proof}

Let $\al \in(\frac{1}{3},\frac{1}{2}]$ and $\X = (1, X, \x)\in \RP$. With the setting in Corollary~\ref{coro:3}, define the map
\begin{equation}
\mathcal{M}:\CC([0, T], W) \to \CC([0, T], W), \quad (Y, Y')\mapsto \left(Y_0+\int_0^{\bullet} F(Y_r)\,d\Z_r, F(Y)Z' \right).
\mlabel{eq:27}
\end{equation}
The fixed point theorem will be carried out in the space of controlled rough paths with prescribed initial values for both the path and its Gubinelli derivative.

In what follows, all paths are considered on $[0,\tau]$, where $\tau<T$ is chosen sufficiently small. 
To specify a closed ball on which $\mathcal M$ acts, it is convenient to fix a center $\mathbf H=(H,H')$ with $H_0=Y_0$. 
A natural choice is
\begin{equation}
H_t:=Y_0+F(Y_0)Z'_0X_{0, t},\quad H'_t:=F(Y_0)Z'_0,\quad 0\le t\le\tau.
\mlabel{eq:28}
\end{equation}

\begin{lemma}
With the setting in Corollary~\ref{coro:3}, the path $\HA =(H,H')$ defined by~(\ref{eq:28}) is an $\X$-controlled rough path.
\end{lemma}

\begin{proof}
Since
\[H_{s, t}=F(Y_0)Z'_0X_{s, t},\quad H'_{s, t}=0,\]
$H$ and $H'$ are both $\alpha$-H\"{o}lder continuous, and $R^{1, \HA}_{s, t} = H'_{s, t}=0$. In addition,
\[R^{0, \HA}_{s, t}:=H_{s, t}-H'_sX_{s, t}=F(Y_0)Z'_0X_{s, t}-F(Y_0)Z'_0X_{s, t}=0,\]
that is, $R^{0, \HA}$ is $2\alpha$-H\"{o}lder continuous. Thus $\HA $ is an $\X$-controlled rough path by Definition~\mref{def:crp}.
\end{proof}

We next introduce the closed set
\begin{equation*}
B_{\tau}(\HA, R):=\Big\{\Y\in \CC([0, \tau], W) \mid Y_0=H_0,\, Y'_0=H'_0,\, |||\Y-\HA|||_{\X; \alpha; [0, \tau]}\le R\Big\}.
\end{equation*}
On $B_{\tau}(\HA, R)$, all controlled paths have the same initial condition. 
Hence the seminorm $\|\cdot\|_{\X; \alpha}$ induces a complete metric equivalent to the restriction of $|||\cdot|||_{\X; \alpha}$. 
We apply Banach fixed point theorem in this metric.
Here the subscript $\tau$ indicates that all paths are restricted to $[0,\tau]$, while the radius $R>0$ will be chosen as needed in the subsequent estimates.

\begin{theorem}
Let $\al \in(\frac{1}{3},\frac{1}{2}]$, $\X =(1, X, \x)\in \D$, $\Z = (Z, Z')\in \CC([0, T], U)$ and $F\in C_b^3(W, \mathcal{L}(U, W))$. For sufficiently small $\tau$,
there exists a unique controlled solution $\Y=(Y, Y')\in \CC([0, \tau], W)$ such that
\begin{equation}
Y_t=Y_0+\int_0^t F(Y_r)\,d\Z_r.
\mlabel{eq:48}
\end{equation}
\mlabel{thm:2}
\end{theorem}

\begin{remark}
Here a controlled solution of~\eqref{eq:48} is understood to satisfy $Y'_t=F(Y_t)Z'_t$.
With this convention,~\eqref{eq:48} is equivalent to
\[(Y_t, Y'_t)=\Y_t=\J_t=\left(Y_0+\int_0^t F(Y_r)\,d\Z_r, F(Y_t)Z'_t \right).\]
\mlabel{rem:add1}
\end{remark}

\begin{proof}[Proof of Theorem~\ref{thm:2}]
We first show
\begin{equation}
\mathcal{M}(B_{\tau}(\HA, R))\subseteq B_{\tau}(\HA, R).
\mlabel{eq:49}
\end{equation}
Let $\Y\in B_{\tau}(\HA, R)$. It follows from~(\ref{eq:27}) and~(\ref{eq:28}) that
\[(\mathcal{M}\Y)_0 =\big(\mathcal{M}(Y, Y')\big)_0=\Big(Y_0, F(Y_0)Z'_0\Big)=(H_0, H'_0)=\HA_0.\]
Notice that $\mathcal{M}\Y-\HA \in \CC$. Then
\begin{align*}
\|\mathcal{M}\Y-\HA\|_{\X; \al}
\le&\ \|\mathcal{M}\Y\|_{\X; \al}+\|\HA\|_{\X; \al}\\
=&\ \|\mathcal{M}\Y\|_{\X; \al} \hspace{1cm}(\text{by $\|\HA\|_{\X; \al}=0 $})\\
=&\ \|\J\|_{\X; \al}
\le (1+C_{\al, T})\Big(\tau^{\al}M(\tau, \|F\|_{C_b^2}, \|\X\|_{\al}, |Y'_0|, |Z'_0|, R, \|\Z\|_{\X; \al})\\
&\ +\|F\|_{C_b^2}(1+|Y'_0|)(1+|Z'_0|)(1+\|\X\|_{\al})^2(1+\|\Z\|_{\X; \al})\Big)\hspace{1cm}(\text{by (\ref{eq:30-1})}).
\end{align*}
Hence
\[\|\mathcal{M}\Y-\HA\|_{\X; \al}\le (1+C_{\al, T})\tau^{\al}M(\tau, \|F\|_{C_b^2}, \|\X\|_{\al}, |Y'_0|, |Z'_0|, R, \|\Z\|_{\X; \al})+\frac{R}{2},\]
where
\[R:=2(1+C_{\al, T})\|F\|_{C_b^2}(1+|Y'_0|)(1+|Z'_0|)(1+\|\X\|_{\al})^2(1+\|\Z\|_{\X; \al}).\]
Taking $\tau$ small enough such that	
\[(1+C_{\al, T})\tau^{\al}M(\tau, \|F\|_{C_b^2}, \|\X\|_{\al}, |Y'_0|, |Z'_0|, R, \|\Z\|_{\X; \al})\le \frac{R}{2},\]
we obtain $\mathcal{M}\Y\in B_{\tau}(\HA, R)$.

Next, for $\Y, \widetilde\Y\in B_{\tau}(\HA, R)$, since $\Y_0 -\widetilde{\Y}_0=\HA_0 -\HA_0 =0$ and $\Y$, $\widetilde\Y$ are controlled rough paths by the same $\X$, it follows from Proposition~\ref{prop:3} that
\begin{align*}
&\ \hspace{-1.3cm}\|\mathcal{M}\Y-\mathcal{M}\widetilde \Y\|_{\X; \al}\\
\le&\  C_{\al, T}\tau^{\al} M(\tau, \|F\|_{C_b^3}, |Y'_0|, |\widetilde Y'_0|, |Z'_0|, \|\Y\|_{\X; \al}, \|\widetilde \Y\|_{\X; \al}, \|\X\|_{\al}, \|\Z\|_{\X; \al})\|\Y-\widetilde \Y\|_{\X; \al}.
\end{align*}
By~(\ref{eq:49}), the above function $M$ is bounded. Choosing again $\tau$ small enough, we obtain
\[\|\mathcal{M}\Y-\mathcal{M}\widetilde \Y\|_{\X; \al}\le \frac{1}{2}\|\Y-\widetilde \Y\|_{\X; \al}.\]
Applying the Banach fixed point theorem, there exists a unique  $\Y\in B_{\tau}(\HA, R)$ such that $\mathcal{M}\Y=\Y$, as required.
\end{proof}

Next, we present the global existence and uniqueness of the solution for~(\ref{eq:48}).
\begin{theorem}
Let $\al \in(\frac{1}{3},\frac{1}{2}]$, $\X =(1, X, \x)\in \D$, $\Z = (Z, Z')\in \CC([0, T], U)$ and $F\in C_b^3(W, \mathcal{L}(U, W))$. There exists a unique controlled solution $\Y=(Y, Y')\in \CC([0, T], W)$ such that
\begin{equation}
Y_t=Y_0+\int_0^t F(Y_r)\,d\Z_r.
\mlabel{eq:48+1}
\end{equation}
\mlabel{thm:3}
\end{theorem}

\begin{proof}
Let $\tau$ be as in Theorem~\ref{thm:2}.
It remains to justify that the local existence time can be chosen uniformly when the solution is restarted. 
In the local fixed point theorem, the time $\tau$ depends on the initial controlled structure only through the initial Gubinelli derivative. 
If the equation is restarted at time \(r\in[0,T]\), then the initial derivative is $Y'_r=F(Y_r)Z'_r$.
Since $F\in C_b^3$, we have
\[
|F(Y_r)Z'_r|\le \|F\|_\infty |Z'_r|.
\]
Moreover,
\[
|Z'_r|\le |Z'_0|+T^\alpha\|Z'\|_\alpha\le |Z'_0|+T^\alpha\|\Z\|_{\X;\alpha}.
\]
Therefore the local existence time can be chosen depending only on
\[
\|F\|_{C_b^3},\quad \|\X\|_\alpha,\quad |Z'_0|,\quad \|\Z\|_{\X;\alpha}, \quad T,
\]
but not on the current value $Y_r $ of the solution at the restarting time.
Hence the local solution can be iterated finitely many times to obtain a solution on the whole interval $[0,T]$.
\end{proof}

\section{Universal limit theorem}\label{ss:sec4}
In this section, we address the robustness of the solution of~(\ref{eq:48+1}) with respect to the initial condition $Y_0$, the driving rough path $\X$ and the driving controlled rough path $\Z$.

\begin{theorem}
Let $\al \in(\frac{1}{3},\frac{1}{2}]$, and let $\X =(1, X, \x)$ and $\widetilde{\X}=(1, \widetilde X, \widetilde \x)$ be in $\D$. Let $\Z = (Z, Z')\in \CC([0, T], U)$, $\widetilde \Z = (\widetilde Z, \widetilde Z')\in \CCA([0, T], U)$, and assume that $F\in C_b^3(W, \mathcal{L}(U, W))$. Let $\Y=(Y, Y')\in \CC([0, T], W )$ $\mathrm{(}$resp. $\widetilde \Y=(\widetilde Y, \widetilde Y')\in \CCA([0, T], W )$$\mathrm{)}$ be the unique solution to the controlled RDE
driven by  $\mathbf Z$ $\mathrm{(}$resp.\ $\widetilde{\mathbf Z}$$\mathrm{)}$:
\begin{equation*}
Y_t=Y_0+\int_0^t F(Y_r)\,d\Z_r\quad \Big(\text{resp.} \widetilde Y_t=\widetilde Y_0+\int_0^t F(\widetilde Y_r)\,d\widetilde \Z_r\Big).
\end{equation*}
Then the estimate
\begin{align*}
&\ \hspace{-0.8cm}d_{\X, \widetilde \X; \al}(\Y, \widetilde \Y)\\
\le&\ C_{\al, T}M(T, \|F\|_{C_b^3}, |Y'_0|, |\widetilde Y'_0|, |Z'_0|, |\widetilde Z'_0|, \|\Y\|_{\X; \al}, \|\widetilde \Y\|_{\widetilde \X; \al}, \|\Z\|_{\X; \al}, \|\widetilde \Z\|_{\widetilde \X; \al}, \|\X\|_{\al}, \|\widetilde \X\|_{\al})\nonumber\\
&\ \times\Big(|Y_0-\widetilde Y_0|+|Y'_0-\widetilde Y'_0|+|Z_0-\widetilde Z_0|+|Z'_0-\widetilde Z'_0|+d_{\X, \widetilde \X; \al}(\Z, \widetilde \Z)+\|\X-\widetilde \X\|_{\al}\Big)
\end{align*}
holds.
\mlabel{thm:4}
\end{theorem}

To prove Theorem~\ref{thm:4}, we first present a refined local stability estimate, which improves Proposition~\ref{prop:3}. 
This estimate will be used on small time intervals, and the global estimate on \([0,T]\) will then follow from the patching lemma stated below.

\begin{proposition}
Let $\al \in(\frac{1}{3},\frac{1}{2}]$, and let $\X =(1, X, \x)$ and $\widetilde{\X}=(1, \widetilde X, \widetilde \x)$ be in $\D$. Let $\Y = (Y, Y')\in \CC([0, T], W)$, $\widetilde{\Y}  = (\widetilde Y, \widetilde Y')\in \CCA([0, T], W)$, $\Z = (Z, Z')\in \CC([0, T], U)$ and $\widetilde \Z = (\widetilde Z, \widetilde Z')\in \CCA([0, T], U)$. Assume that $F\in C_b^3(W, \mathcal{L}(U, W))$. Then
\begin{align*}
&\ \hspace{-0.8cm}d_{\X, \widetilde \X; \al}(\J, \widetilde \J)\\
\le&\  C_{\al, T}M(T, \|F\|_{C_b^3}, |Y'_0|, |\widetilde Y'_0|, |Z'_0|, |\widetilde Z'_0|, \|\Y\|_{\X; \al}, \|\widetilde \Y\|_{\widetilde \X; \al}, \|\Z\|_{\X; \al}, \|\widetilde \Z\|_{\widetilde \X; \al}, \|\X\|_{\al}, \|\widetilde \X\|_{\al})\nonumber\\
&\ \times\Big(T^{\al}d_{\X, \widetilde \X; \al}(\Y, \widetilde \Y)+|Y_0-\widetilde Y_0|+|Y'_0-\widetilde Y'_0|\\
&\ \hspace{0.6cm}+d_{\X, \widetilde \X; \al}(\Z, \widetilde \Z)+|Z_0-\widetilde Z_0|+|Z'_0-\widetilde Z'_0|+\|\X-\widetilde \X\|_{\al}\Big),
\end{align*}
where
\[
\J=\Big( Y_0+\int_0^{\bullet} F( Y_r)\,d \Z_r, F( Y) Z' \Big),\quad
\widetilde \J=\Big(\widetilde Y_0+\int_0^{\bullet} F(\widetilde Y_r)\,d\widetilde \Z_r, F(\widetilde Y)\widetilde Z' \Big).
\]
\mlabel{prop:4}
\end{proposition}

\begin{proof}
We use the methods in Lemma~\ref{thm:1} and Proposition~\ref{prop:3} to prove this conclusion, and apply the point removal method to the rough integral $\int_s^t Y_r\,d\Z_r-\int_s^t \widetilde Y_r\,d\widetilde \Z_r$. 
By the composition estimates for controlled rough paths, it suffices to establish the corresponding estimate for general controlled rough paths $\Y$ and $\widetilde\Y$. The desired result then follows by applying this estimate to $F(\Y)$ and $F(\widetilde \Y)$.
By applying~(\ref{eq:20+1}),
\begin{align}
&\ \hspace{-1.6cm}\Big(\int_PY_r\,d\Z_r-\int_P\widetilde Y_r\,d\widetilde \Z_r\Big)-\Big(\int_{P\setminus \{t_j\}}Y_r\,d\Z_r-\int_{P\setminus \{t_j\}}\widetilde Y_r\,d\widetilde \Z_r\Big)\nonumber\\
=&\ \Big(R^{0, \Y}_{t_{j-1}, t_{j}}R^{0, \Z}_{t_{j}, t_{j+1}}-R^{0, \widetilde\Y}_{t_{j-1}, t_{j}}R^{0, \widetilde\Z}_{t_{j}, t_{j+1}}\Big)+\Big(R^{0, \Y}_{t_{j-1}, t_{j}}Z'_{t_{j}}X_{t_{j}, t_{j+1}}-R^{0, \widetilde \Y}_{t_{j-1}, t_{j}}\widetilde Z'_{t_{j}}\widetilde X_{t_{j}, t_{j+1}}\Big)\nonumber\\
&\ +\Big(Y'_{t_{j-1}}X_{t_{j-1}, t_{j}}R^{0, \Z}_{t_{j}, t_{j+1}}-\widetilde Y'_{t_{j-1}}\widetilde X_{t_{j-1}, t_{j}}R^{0, \widetilde \Z}_{t_{j}, t_{j+1}}\Big)\mlabel{eq:51}\\
&\ +\Big(Y'_{t_{j-1}}Z'_{t_{j-1}, t_{j}}X_{t_{j-1}, t_{j}}\otimes X_{t_{j}, t_{j+1}}-\widetilde Y'_{t_{j-1}}\widetilde Z'_{t_{j-1}, t_{j}}\widetilde X_{t_{j-1}, t_{j}}\otimes \widetilde X_{t_{j}, t_{j+1}}\Big)\nonumber\\
&\ +\Big(Y'_{t_{j-1}, t_{j}}Z'_{t_{j}}\x_{t_{j}, t_{j+1}}-\widetilde Y'_{t_{j-1}, t_{j}}\widetilde Z'_{t_{j}}\widetilde \x_{t_{j}, t_{j+1}}\Big)\nonumber\\
&\ +\Big(Y'_{t_{j-1}}Z'_{t_{j-1}, t_{j}}\x_{t_{j}, t_{j+1}}-\widetilde Y'_{t_{j-1}}\widetilde Z'_{t_{j-1}, t_{j}}\widetilde \x_{t_{j}, t_{j+1}}\Big).\nonumber
\end{align}
We estimate the six terms separately on the right-hand side of the equation.

\noindent{\bf Step 1. } For the first term,
\begin{align}
&\ \hspace{-1.9cm}\big|R^{0, \Y}_{t_{j-1}, t_{j}}R^{0, \Z}_{t_{j}, t_{j+1}}-R^{0, \widetilde\Y}_{t_{j-1}, t_{j}}R^{0, \widetilde\Z}_{t_{j}, t_{j+1}}\big|\nonumber\\
\le&\  \big|R^{0, \Y}_{t_{j-1}, t_{j}}-R^{0, \widetilde\Y}_{t_{j-1}, t_{j}}\big|\,\big|R^{0, \Z}_{t_{j}, t_{j+1}}\big|+\big|R^{0, \widetilde\Y}_{t_{j-1}, t_{j}}\big|\,\big|R^{0, \Z}_{t_{j}, t_{j+1}}-R^{0, \widetilde\Z}_{t_{j}, t_{j+1}}\big| \nonumber\\
\le&\  \big\|R^{0, \Y}-R^{0, \widetilde\Y}\big\|_{2\al}\big\|R^{0, \Z}\big\|_{2\al}|t_{j}-t_{j-1}|^{2\al}|t_{j+1}-t_{j}|^{2\al} \nonumber\\
&\ +\big\|R^{0, \widetilde\Y}\big\|_{2\al}\big\|R^{0, \Z}-R^{0, \widetilde\Z}\big\|_{2\al}|t_{j}-t_{j-1}|^{2\al}|t_{j+1}-t_{j}|^{2\al} \mlabel{eq:52}\\
\le&\ d_{\X, \widetilde \X; \al}(\Y, \widetilde \Y)\|\Z\|_{\X;\al}|t_{j}-t_{j-1}|^{2\al}|t_{j+1}-t_{j}|^{2\al}\nonumber\\
&\ +\|\widetilde \Y\|_{\widetilde \X; \al}\,d_{\X, \widetilde \X; \al}(\Z, \widetilde \Z)|t_{j}-t_{j-1}|^{2\al}|t_{j+1}-t_{j}|^{2\al}\hspace{1cm}(\text{by (\ref{eq:norm2}) and (\ref{eq:norm1})}). \nonumber
\end{align}

\noindent{\bf Step 2. } For the second term,
\begin{align}
&\ \hspace{-1.5cm}\big|R^{0, \Y}_{t_{j-1}, t_{j}}Z'_{t_{j}}X_{t_{j}, t_{j+1}}-R^{0, \widetilde \Y}_{t_{j-1}, t_{j}}\widetilde Z'_{t_{j}}\widetilde X_{t_{j}, t_{j+1}}\big| \nonumber\\
\le&\  \big|R^{0, \Y}_{t_{j-1}, t_{j}}-R^{0, \widetilde \Y}_{t_{j-1}, t_{j}}\big|\, \big|Z'_{t_{j}}\big|\,\big|X_{t_{j}, t_{j+1}}\big|+\big|R^{0, \widetilde \Y}_{t_{j-1}, t_{j}}\big|\,\big|Z'_{t_{j}}X_{t_{j}, t_{j+1}}-\widetilde Z'_{t_{j}}\widetilde X_{t_{j}, t_{j+1}}\big| \nonumber\\
\le&\  \big|R^{0, \Y}_{t_{j-1}, t_{j}}-R^{0, \widetilde \Y}_{t_{j-1}, t_{j}}\big|\, \big|Z'_{t_{j}}\big|\,\big|X_{t_{j}, t_{j+1}}\big|
+\big|R^{0, \widetilde \Y}_{t_{j-1}, t_{j}}\big|\,\big|Z'_{t_{j}}-\widetilde Z'_{t_{j}}\big|\,\big|X_{t_{j}, t_{j+1}}\big|\nonumber\\
&\ +\big|R^{0, \widetilde \Y}_{t_{j-1}, t_{j}}\big|\,\big|\widetilde Z'_{t_{j}}\big|\,\big|X_{t_{j}, t_{j+1}}-\widetilde X_{t_{j}, t_{j+1}}\big| \nonumber\\
\le&\  \big\|R^{0, \Y}-R^{0, \widetilde\Y}\big\|_{2\al}\|Z'\|_{\infty}\|X\|_{\al}|t_{j}-t_{j-1}|^{2\al}|t_{j+1}-t_{j}|^{\al}\nonumber\\
&\ + \big\|R^{0, \widetilde\Y}\big\|_{2\al}\|Z'-\widetilde Z'\|_{\infty}\|X\|_{\al}|t_{j}-t_{j-1}|^{2\al}|t_{j+1}-t_{j}|^{\al} \nonumber\\
&\ +\big\|R^{0, \widetilde\Y}\big\|_{2\al}\|\widetilde Z'\|_{\infty}\|X-\widetilde X\|_{\al}|t_{j}-t_{j-1}|^{2\al}|t_{j+1}-t_{j}|^{\al}   \nonumber\\
\le&\  \big\|R^{0, \Y}-R^{0, \widetilde\Y}\big\|_{2\al}\Big(T^\al\|Z'\|_{\al}+|Z'_0|\Big)\|X\|_{\al}|t_{j}-t_{j-1}|^{2\al}|t_{j+1}-t_{j}|^{\al} \mlabel{eq:53}\\
&\ + \big\|R^{0, \widetilde\Y}\big\|_{2\al}\Big(T^\al\|Z'-\widetilde Z'\|_{\al}+|Z'_0-\widetilde Z'_0|\Big)\|X\|_{\al}|t_{j}-t_{j-1}|^{2\al}|t_{j+1}-t_{j}|^{\al} \nonumber\\
&\ +\big\|R^{0, \widetilde\Y}\big\|_{2\al}\Big(T^\al\|\widetilde Z'\|_{\al}+|\widetilde Z'_0|\Big)\|X-\widetilde X\|_{\al}|t_{j}-t_{j-1}|^{2\al}|t_{j+1}-t_{j}|^{\al} \quad(\text{by (\ref{eq:6})})   \nonumber\\
\le&\  d_{\X, \widetilde \X; \al}(\Y, \widetilde \Y)\Big(T^\al\|\Z\|_{\X; \al}+|Z'_0|\Big)\|\X\|_{\al}|t_{j}-t_{j-1}|^{2\al}|t_{j+1}-t_{j}|^{\al}\nonumber\\
&\ + \|\widetilde\Y\|_{\widetilde\X; \al}\Big(T^{\al}d_{\X, \widetilde \X; \al}(\Z, \widetilde \Z)+|Z'_0-\widetilde Z'_0|\Big)\|\X\|_{\al}|t_{j}-t_{j-1}|^{2\al}|t_{j+1}-t_{j}|^{\al} \nonumber\\
&\ +\|\widetilde\Y\|_{\widetilde\X; \al}\Big(T^\al\|\widetilde\Z\|_{\widetilde\X; \al}+|\widetilde Z'_0|\Big)\|\X-\widetilde \X\|_{\al}|t_{j}-t_{j-1}|^{2\al}|t_{j+1}-t_{j}|^{\al} \nonumber\\
&\ \hspace{9cm} (\text{by~(\ref{eq:note4}),~(\ref{eq:norm2}) and~(\ref{eq:norm1})}).\nonumber
\end{align}

\noindent{\bf Step 3. } For the third term,
\begin{align}
&\  \hspace{-1.65cm}\big|Y'_{t_{j-1}}X_{t_{j-1}, t_{j}}R^{0, \Z}_{t_{j}, t_{j+1}}-\widetilde Y'_{t_{j-1}}\widetilde X_{t_{j-1}, t_{j}}R^{0, \widetilde \Z}_{t_{j}, t_{j+1}} \big|  \nonumber\\
\le&\ |Y'_{t_{j-1}}-\widetilde Y'_{t_{j-1}}|\,|X_{t_{j-1}, t_{j}}|\,\big|R^{0, \Z}_{t_{j}, t_{j+1}} \big|+|\widetilde Y'_{t_{j-1}}|\,\big|X_{t_{j-1}, t_{j}}R^{0, \Z}_{t_{j}, t_{j+1}}-\widetilde X_{t_{j-1}, t_{j}}R^{0, \widetilde \Z}_{t_{j}, t_{j+1}}\big| \nonumber\\
\le&\ |Y'_{t_{j-1}}-\widetilde Y'_{t_{j-1}}|\,|X_{t_{j-1}, t_{j}}|\,\big|R^{0, \Z}_{t_{j}, t_{j+1}} \big|
+|\widetilde Y'_{t_{j-1}}|\,|X_{t_{j-1}, t_{j}}-\widetilde X_{t_{j-1}, t_{j}}|\,\big|R^{0, \Z}_{t_{j}, t_{j+1}}|\nonumber\\
&\ +|\widetilde Y'_{t_{j-1}}|\,|\widetilde X_{t_{j-1}, t_{j}}|\,\big|R^{0, \Z}_{t_{j}, t_{j+1}}-R^{0, \widetilde \Z}_{t_{j}, t_{j+1}}\big| \nonumber\\
\le&\ \|Y'-\widetilde Y'\|_{\infty}\|X\|_{\al} \big\|R^{0, \Z}\big\|_{2\al}|t_{j}-t_{j-1}|^{\al}|t_{j+1}-t_{j}|^{2\al}\nonumber\\
&\ +\|\widetilde Y'\|_{\infty}\|X-\widetilde X\|_{\al}\big\|R^{0, \Z}\big\|_{2\al}|t_{j}-t_{j-1}|^{\al}|t_{j+1}-t_{j}|^{2\al}\nonumber\\
&\ +\|\widetilde Y'\|_{\infty}\|\widetilde X\|_{\al}\big\|R^{0, \Z}-R^{0, \widetilde \Z}\big\|_{2\al}|t_{j}-t_{j-1}|^{\al}|t_{j+1}-t_{j}|^{2\al} \mlabel{eq:54}\\
\le&\ \Big(T^\al\|Y'-\widetilde Y'\|_{\al}+|Y'_0-\widetilde Y'_0|\Big)\|X\|_{\al} \big\|R^{0, \Z}\big\|_{2\al}|t_{j}-t_{j-1}|^{\al}|t_{j+1}-t_{j}|^{2\al} \nonumber\\
&\ +\Big(T^\al\|\widetilde Y'\|_{\al}+|\widetilde Y'_0|\Big)\|X-\widetilde X\|_{\al}\big\|R^{0, \Z}\big\|_{2\al}|t_{j}-t_{j-1}|^{\al}|t_{j+1}-t_{j}|^{2\al} \nonumber\\
&\ +\Big(T^\al\|\widetilde Y'\|_{\al}+|\widetilde Y'_0|\Big)\|\widetilde X\|_{\al}\big\|R^{0, \Z}-R^{0, \widetilde \Z}\big\|_{2\al}|t_{j}-t_{j-1}|^{\al}|t_{j+1}-t_{j}|^{2\al}  \  (\text{by~(\ref{eq:6})})\nonumber\\
\le&\ \Big(T^\al d_{\X, \widetilde \X; \al}(\Y, \widetilde \Y)+|Y'_0-\widetilde Y'_0|\Big)\|\X\|_{\al} \|\Z\|_{\X; \al}|t_{j}-t_{j-1}|^{\al}|t_{j+1}-t_{j}|^{2\al} \nonumber\\
&\ +\Big(T^\al\|\widetilde \Y\|_{\widetilde\X; \al}+|\widetilde Y'_0|\Big)\|\X-\widetilde \X\|_{\al}\|\Z\|_{\X; \al}|t_{j}-t_{j-1}|^{\al}|t_{j+1}-t_{j}|^{2\al} \nonumber\\
&\ +\Big(T^\al\|\widetilde \Y\|_{\widetilde\X; \al}+|\widetilde Y'_0|\Big)\|\widetilde \X\|_{\al}d_{\X, \widetilde \X; \al}(\Z, \widetilde \Z)|t_{j}-t_{j-1}|^{\al}|t_{j+1}-t_{j}|^{2\al}. \nonumber
\end{align}

\noindent{\bf Step 4. } For the fourth term,
\begin{align}
&\  \hspace{-2.2cm}|Y'_{t_{j-1}}Z'_{t_{j-1}, t_{j}}X_{t_{j-1}, t_{j}}\otimes X_{t_{j}, t_{j+1}}-\widetilde Y'_{t_{j-1}}\widetilde Z'_{t_{j-1}, t_{j}}\widetilde X_{t_{j-1}, t_{j}}\otimes \widetilde X_{t_{j}, t_{j+1}}|\nonumber\\
\le&\ |Y'_{t_{j-1}}-\widetilde Y'_{t_{j-1}}|\,|Z'_{t_{j-1}, t_{j}}|\,|X_{t_{j-1}, t_{j}}|\,|X_{t_{j}, t_{j+1}}|\nonumber\\
&\ +|\widetilde Y'_{t_{j-1}}|\,|Z'_{t_{j-1}, t_{j}}-\widetilde Z'_{t_{j-1}, t_{j}}|\,|X_{t_{j-1}, t_{j}}|\,|X_{t_{j}, t_{j+1}}| \nonumber\\
&\ +|\widetilde Y'_{t_{j-1}}|\,|\widetilde Z'_{t_{j-1}, t_{j}}|\,|X_{t_{j-1}, t_{j}}-\widetilde X_{t_{j-1}, t_{j}}|\,|X_{t_{j}, t_{j+1}}|\nonumber\\
&\ +|\widetilde Y'_{t_{j-1}}|\,|\widetilde Z'_{t_{j-1}, t_{j}}|\,|\widetilde X_{t_{j-1}, t_{j}}|\,|X_{t_{j}, t_{j+1}}-\widetilde X_{t_{j}, t_{j+1}}| \nonumber\\
\le&\ \|Y'-\widetilde Y'\|_{\infty}\|Z'\|_{\al}\|X\|_{\al}^2|t_{j}-t_{j-1}|^{2\al}|t_{j+1}-t_{j}|^{\al}\nonumber\\
&\ +\|\widetilde Y'\|_{\infty}\|Z'-\widetilde Z'\|_{\al}\|X\|_{\al}^2|t_{j}-t_{j-1}|^{2\al}|t_{j+1}-t_{j}|^{\al} \nonumber\\
&\ +\|\widetilde Y'\|_{\infty}\|\widetilde Z'\|_{\al}\|X-\widetilde X\|_{\al}\|X\|_{\al}|t_{j}-t_{j-1}|^{2\al}|t_{j+1}-t_{j}|^{\al} \nonumber\\
&\ +\|\widetilde Y'\|_{\infty}\|\widetilde Z'\|_{\al}\|\widetilde X\|_{\al}\|X-\widetilde X\|_{\al}|t_{j}-t_{j-1}|^{2\al}|t_{j+1}-t_{j}|^{\al} \mlabel{eq:55}\\
\le&\ \Big(T^\al\|Y'-\widetilde Y'\|_{\al}+|Y'_0-\widetilde Y'_0|\Big)\|Z'\|_{\al}\|X\|_{\al}^2|t_{j}-t_{j-1}|^{2\al}|t_{j+1}-t_{j}|^{\al} \nonumber\\
&\ +\Big(T^\al\|\widetilde Y'\|_{\al}+|\widetilde Y'_0|\Big)\|Z'-\widetilde Z'\|_{\al}\|X\|_{\al}^2|t_{j}-t_{j-1}|^{2\al}|t_{j+1}-t_{j}|^{\al} \nonumber\\
&\ +\Big(T^\al\|\widetilde Y'\|_{\al}+|\widetilde Y'_0|\Big)\|\widetilde Z'\|_{\al}\|X-\widetilde X\|_{\al}\|X\|_{\al}|t_{j}-t_{j-1}|^{2\al}|t_{j+1}-t_{j}|^{\al} \nonumber\\
&\ +\Big(T^\al\|\widetilde Y'\|_{\al}+|\widetilde Y'_0|\Big)\|\widetilde Z'\|_{\al}\|\widetilde X\|_{\al}\|X-\widetilde X\|_{\al}|t_{j}-t_{j-1}|^{2\al}|t_{j+1}-t_{j}|^{\al}  \nonumber\\
\le&\ \Big(T^\al d_{\X, \widetilde \X; \al}(\Y, \widetilde \Y)+|Y'_0-\widetilde Y'_0|\Big)\|\Z\|_{\X; \al}\|\X\|_{\al}^2|t_{j}-t_{j-1}|^{2\al}|t_{j+1}-t_{j}|^{\al} \nonumber\\
&\ +\Big(T^\al\|\widetilde \Y\|_{\widetilde\X; \al}+|\widetilde Y'_0|\Big)d_{\X, \widetilde \X; \al}(\Z, \widetilde \Z)\|\X\|_{\al}^2|t_{j}-t_{j-1}|^{2\al}|t_{j+1}-t_{j}|^{\al} \nonumber\\
&\ +\Big(T^\al\|\widetilde \Y\|_{\widetilde\X; \al}+|\widetilde Y'_0|\Big)\|\widetilde \Z\|_{\widetilde\X; \al}\|\X-\widetilde \X\|_{\al}\|X\|_{\al}|t_{j}-t_{j-1}|^{2\al}|t_{j+1}-t_{j}|^{\al} \nonumber\\
&\ +\Big(T^\al\|\widetilde \Y\|_{\widetilde\X; \al}+|\widetilde Y'_0|\Big)\|\widetilde \Z\|_{\widetilde\X; \al}\|\widetilde \X\|_{\al}\|\X-\widetilde \X\|_{\al}|t_{j}-t_{j-1}|^{2\al}|t_{j+1}-t_{j}|^{\al}. \nonumber
\end{align}

\noindent{\bf Step 5. } For the fifth term,
\begin{align}
&\ \hspace{-2.25cm}|Y'_{t_{j-1}, t_{j}}Z'_{t_{j}}\x_{t_{j}, t_{j+1}}-\widetilde Y'_{t_{j-1}, t_{j}}\widetilde Z'_{t_{j}}\widetilde \x_{t_{j}, t_{j+1}}| \nonumber\\
\le&\ |Y'_{t_{j-1}, t_{j}}-\widetilde Y'_{t_{j-1}, t_{j}}|\,|Z'_{t_{j}}|\,|\x_{t_{j}, t_{j+1}}|+|\widetilde Y'_{t_{j-1}, t_{j}}|\,|Z'_{t_{j}}-\widetilde Z'_{t_{j}}|\,|\x_{t_{j}, t_{j+1}}|\nonumber\\
&\ +|\widetilde Y'_{t_{j-1}, t_{j}}|\,|\widetilde Z'_{t_{j}}|\,|\x_{t_{j}, t_{j+1}}-\widetilde \x_{t_{j}, t_{j+1}}| \nonumber\\
\le&\ \|Y'-\widetilde Y'\|_{\al}\|Z'\|_{\infty}\|\x\|_{2\al}|t_{j}-t_{j-1}|^{\al}|t_{j+1}-t_{j}|^{2\al}\nonumber\\
&\ +\|\widetilde Y'\|_{\al}\|Z'-\widetilde Z'\|_{\infty}\|\x\|_{2\al}|t_{j}-t_{j-1}|^{\al}|t_{j+1}-t_{j}|^{2\al} \nonumber\\
&\ +\|\widetilde Y'\|_{\al}\|\widetilde Z'\|_{\infty}\|\x-\widetilde \x\|_{2\al}|t_{j}-t_{j-1}|^{\al}|t_{j+1}-t_{j}|^{2\al}  \mlabel{eq:56}\\
\le&\ d_{\X, \widetilde \X; \al}(\Y, \widetilde \Y)\Big(T^\al\|\Z\|_{\X; \al}+| Z'_0|\Big)\|\X\|_{\al}|t_{j}-t_{j-1}|^{\al}|t_{j+1}-t_{j}|^{2\al}\nonumber\\
&\ +\|\widetilde \Y\|_{\widetilde\X; \al}\Big(T^\al d_{\X, \widetilde \X; \al}(\Z, \widetilde \Z)+|Z'_0-\widetilde Z'_0|\Big)\|\X\|_{\al}|t_{j}-t_{j-1}|^{\al}|t_{j+1}-t_{j}|^{2\al} \nonumber\\
&\ +\|\widetilde \Y\|_{\widetilde\X; \al}\Big(T^\al\|\widetilde \Z\|_{\widetilde\X; \al}+|\widetilde Z'_0|\Big)\|\X-\widetilde \X\|_{\al}|t_{j}-t_{j-1}|^{\al}|t_{j+1}-t_{j}|^{2\al}.\nonumber
\end{align}

\noindent{\bf Step 6. } For the sixth term,
\begin{align}
&\ \hspace{-2.5cm}|Y'_{t_{j-1}}Z'_{t_{j-1}, t_{j}}\x_{t_{j}, t_{j+1}}-\widetilde Y'_{t_{j-1}}\widetilde Z'_{t_{j-1}, t_{j}}\widetilde \x_{t_{j}, t_{j+1}}| \nonumber\\
\le&\ |Y'_{t_{j-1}}-\widetilde Y'_{t_{j-1}}|\,|Z'_{t_{j-1}, t_{j}}|\,|\x_{t_{j}, t_{j+1}}|+|\widetilde Y'_{t_{j-1}}|\,|Z'_{t_{j-1}, t_{j}}-\widetilde Z'_{t_{j-1}, t_{j}}|\,|\x_{t_{j}, t_{j+1}}|\nonumber\\
&\ +|\widetilde Y'_{t_{j-1}}|\,|\widetilde Z'_{t_{j-1}, t_{j}}|\,|\x_{t_{j}, t_{j+1}}-\widetilde \x_{t_{j}, t_{j+1}}| \nonumber\\
\le&\ \|Y'-\widetilde Y'\|_{\infty}\|Z'\|_{\al}\|\x\|_{2\al}|t_{j}-t_{j-1}|^{\al}|t_{j+1}-t_{j}|^{2\al} \nonumber\\
&\ +\|\widetilde Y'\|_{\infty}\|Z'-\widetilde Z'\|_{\al}\|\x\|_{2\al}|t_{j}-t_{j-1}|^{\al}|t_{j+1}-t_{j}|^{2\al} \mlabel{eq:57}\\
&\ +\|\widetilde Y'\|_{\infty}\|\widetilde Z'\|_{\al}\|\x-\widetilde \x\|_{2\al}|t_{j}-t_{j-1}|^{\al}|t_{j+1}-t_{j}|^{2\al}\nonumber\\
\le&\ \Big(T^\al d_{\X, \widetilde \X; \al}(\Y, \widetilde \Y)+|Y'_0-\widetilde Y'_0|\Big)\|\Z\|_{\X; \al}\|\X\|_{\al}|t_{j}-t_{j-1}|^{\al}|t_{j+1}-t_{j}|^{2\al} \nonumber\\
&\ +\Big(T^\al\|\widetilde \Y\|_{\widetilde\X; \al}+|\widetilde Y'_0|\Big)d_{\X, \widetilde \X; \al}(\Z, \widetilde \Z)\|\X\|_{\al}|t_{j}-t_{j-1}|^{\al}|t_{j+1}-t_{j}|^{2\al} \nonumber\\
&\ +\Big(T^\al\|\widetilde \Y\|_{\widetilde\X; \al}+|\widetilde Y'_0|\Big)\|\widetilde \Z\|_{\widetilde\X; \al}\|\X-\widetilde \X\|_{\al}|t_{j}-t_{j-1}|^{\al}|t_{j+1}-t_{j}|^{2\al}.  \nonumber
\end{align}

Substituting~(\ref{eq:52})--(\ref{eq:57}) into~(\ref{eq:51}), we conclude
\begin{align*}
&\ \hspace{-1.5cm}\Big|\Big(\int_PY_r\,d\Z_r-\int_P\widetilde Y_r\,d\widetilde \Z_r\Big)-\Big(\int_{P\setminus \{t_j\}}Y_r\,d\Z_r-\int_{P\setminus \{t_j\}}\widetilde Y_r\,d\widetilde \Z_r\Big)\Big|\nonumber\\
\le&\ \Big(1+T^\al \Big) \Big(1+|Y_0'|+\|\Y\|_{\X; \al} \Big) \Big(1+|\widetilde Y_0'|+\|\widetilde \Y\|_{\widetilde\X; \al} \Big)\Big(1+|Z_0'|+\|\Z\|_{\X; \al} \Big)\\
&\ \times\Big(1+|\widetilde Z_0'|+\|\widetilde \Z\|_{\widetilde\X; \al} \Big)\Big(1+\|\X\|_{\al}+\|\widetilde\X\|_{\al} \Big)^2 \nonumber\\
&\ \times\Big(d_{\X, \widetilde \X; \al}(\Y, \widetilde \Y)+d_{\X, \widetilde \X; \al}(\Z, \widetilde \Z)+|Y_0'-\widetilde Y_0'|+|Z_0'-\widetilde Z_0'|+\|\X-\widetilde \X\|_{\al} \Big)\\
&\ \times |t_{j+1}-t_{j-1}|^{3\al}.
\end{align*}
Now that we have derived a relation analogous to~(\ref{eq:14}), we can similarly obtain a counterpart to~(\ref{eq:15}), and so
\begin{equation}
\begin{aligned}
&\ \hspace{-0.8cm}\Big|\Big(\int_s^tY_r\,d\Z_r-\int_s^t\widetilde Y_r\,d\widetilde \Z_r\Big)-\Big((Y_sZ_{s, t}+Y'_sZ'_s\x_{s, t})-(\widetilde Y_s\widetilde Z_{s, t}+\widetilde Y'_s\widetilde Z'_s\widetilde \x_{s, t})\Big)\Big|\\
\le&\ C_{\al, T} \Big(1+|Y_0'|+\|\Y\|_{\X; \al} \Big) \Big(1+|\widetilde Y_0'|+\|\widetilde \Y\|_{\widetilde\X; \al} \Big)\Big(1+|Z_0'|+\|\Z\|_{\X; \al} \Big) \\
&\ \times \Big(1+|\widetilde Z_0'|+\|\widetilde \Z\|_{\widetilde\X; \al} \Big) \Big(1+\|\X\|_{\al}+\|\widetilde\X\|_{\al}\Big)^2 \\
&\ \times\Big(d_{\X, \widetilde \X; \al}(\Y, \widetilde \Y)+d_{\X, \widetilde \X; \al}(\Z, \widetilde \Z)+|Y_0'-\widetilde Y_0'|+|Z_0'-\widetilde Z_0'|+\|\X-\widetilde \X\|_{\al} \Big)|t-s|^{3\al}.
\end{aligned}
\mlabel{eq:59}
\end{equation}
Finally, by applying the method in the proof of Proposition~\ref{prop:3} and combining with~(\ref{eq:59}), we obtain the desired result.
\end{proof}

\begin{lemma}\label{lem:patching}
Let $\{0=t_0<\cdots<t_N=T\}$ be a finite partition. 
Suppose $\Y$ is an $\X$-controlled rough path and $\widetilde\Y$ is a $\widetilde{\X}$-controlled rough path such that, on each subinterval $[t_i,t_{i+1}]$,
\[
d_{\X,\widetilde{\X};\alpha,[t_i,t_{i+1}]}(\Y,\widetilde{\Y})\le \varepsilon_i.
\]
If, for any $ t_i$, $|Y_{t_i}-\widetilde Y_{t_i}|+|Y'_{t_i}-\widetilde Y'_{t_i}|$ are bounded by $\varepsilon$, then
\[
d_{\X,\widetilde{\X};\alpha,[0,T]}(\Y,\widetilde{\Y})\le C_T\Bigg(\sum_{i=0}^{N-1}\varepsilon_i+\varepsilon+\|\X-\widetilde{\X}\|_{\alpha,[0,T]}\Bigg),
\]
where $C_T$ depends only on $T,\alpha,N$ and the a priori controlled norms.
\end{lemma}

\begin{proof}
This follows by decomposing every interval $[s,t]\subset[0,T]$ along the partition points and using the controlled expansions on each subinterval. 
The cross terms are controlled by the endpoint errors and by $\|\X-\widetilde{\X}\|_{\alpha,[0,T]}$. 
Since the partition is finite, the local estimates patch together with a constant depending only on the partition and the a priori bounds.
\end{proof}

With the local stability estimate and the patching lemma in hand, we can now turn to the proof of the universal limit theorem.

\begin{proof}[Proof of Theorem~\ref{thm:4}]
Set
\[
\mathcal E:=|Y_0-\widetilde Y_0|+|Y'_0-\widetilde Y'_0|+|Z_0-\widetilde Z_0|+|Z'_0-\widetilde Z'_0|
+d_{\X,\widetilde{\X};\alpha}(\Z,\widetilde{\Z})+\|\X-\widetilde{\X}\|_\alpha .
\]
We first prove the desired estimate on a sufficiently small interval. 
By Remark~\ref{rem:add1} we have $\Y=\J$ and $\widetilde{\Y}=\widetilde{\J}$. Applying Proposition~\ref{prop:4} on an interval $[s,t]\subset[0,T]$, with $t-s\leq \tau$, we obtain
\begin{align*}
d_{\X,\widetilde{\X};\alpha,[s,t]}(\Y,\widetilde{\Y})
\le&\ C_{\al, T}M_0\Big((t-s)^\al d_{\X,\widetilde{\X};\alpha,[s,t]}(\Y,\widetilde{\Y})+|Y_s-\widetilde Y_s|+|Y'_s-\widetilde Y'_s|\\
&\ +|Z_s-\widetilde Z_s|+|Z'_s-\widetilde Z'_s| +d_{\X,\widetilde{\X};\alpha,[s,t]}(\Z,\widetilde{\Z})+\|\X-\widetilde{\X}\|_{\alpha,[s,t]}
\Big),
\end{align*}
where $M_0$ depends only on the a priori quantities appearing in the statement of the theorem, namely
\[
T,\ \|F\|_{C_b^3},\ |Y'_0|,\ |\widetilde Y'_0|,\ |Z'_0|,\ |\widetilde Z'_0|,
\ \|\Y\|_{\X;\alpha},\ \|\widetilde{\Y}\|_{\widetilde{\X};\alpha},
\ \|\Z\|_{\X;\alpha},\ \|\widetilde{\Z}\|_{\widetilde{\X};\alpha},
\ \|\X\|_\alpha,\ \|\widetilde{\X}\|_\alpha .
\]
Choose $\tau$ sufficiently small such that
\[
C_{\al, T} M_0 \tau^\alpha \leq \frac12 .
\]
Then, for every interval $[s,t]$ with $t-s\le \tau$, the preceding
estimate implies
\begin{equation}
\begin{aligned}
d_{\X,\widetilde{\X};\alpha,[s,t]}(\Y,\widetilde{\Y})
\le&\ 2C_{\al, T} M_0\Big(|Y_s-\widetilde Y_s|+|Y'_s-\widetilde Y'_s|  \\
&\ +|Z_s-\widetilde Z_s|+|Z'_s-\widetilde Z'_s|+d_{\X,\widetilde{\X};\alpha,[s,t]}(\Z,\widetilde{\Z})+\|\X-\widetilde{\X}\|_{\alpha,[s,t]}
\Big).
\end{aligned}
\label{eq:local-ult-estimate}
\end{equation}

We now estimate the endpoint errors. 
By the definitions of $\Z$ and $\widetilde{\Z}$, 
\begin{equation}
|Z_s-\widetilde Z_s|+|Z'_s-\widetilde Z'_s|\le C_T\Big(|Z_0-\widetilde Z_0|+|Z'_0-\widetilde Z'_0|
+d_{\X,\widetilde{\X};\alpha}(\Z,\widetilde{\Z})+\|\X-\widetilde{\X}\|_\alpha\Big).
\label{eq:driver-endpoint-control}
\end{equation}
Indeed, this follows by writing
\[
Z_{0,s}=Z'_0X_{0,s}+R^{0,\Z}_{0,s},\quad \widetilde Z_{0,s}=\widetilde Z'_0\widetilde X_{0,s}+R^{0,\widetilde{\Z}}_{0,s},
\]
and similarly
\[
Z'_s-Z'_0=R^{1,\Z}_{0,s},\quad\widetilde Z'_s-\widetilde Z'_0=R^{1,\widetilde{\Z}}_{0,s}.
\]

Let $0=t_0<t_1<\cdots<t_N=T$ be a partition such that $t_{i+1}-t_i\leq\tau$. We claim that
\begin{equation}
\max_{0\le i\le N}\Big(|Y_{t_i}-\widetilde Y_{t_i}|+|Y'_{t_i}-\widetilde Y'_{t_i}|\Big)\le C_T\mathcal E .
\label{eq:endpoint-control}
\end{equation}
For $i=0$, this is part of the definition of $\mathcal E$. 
Suppose it holds at time $t_i$. 
Applying \eqref{eq:local-ult-estimate} on $[t_i,t_{i+1}]$, and using \eqref{eq:driver-endpoint-control}, gives
\begin{equation}
d_{\X,\widetilde{\X};\alpha,[t_i,t_{i+1}]}(\Y,\widetilde{\Y})\le C_T
\left(|Y_{t_i}-\widetilde Y_{t_i}|+|Y'_{t_i}-\widetilde Y'_{t_i}|+\mathcal E\right).
\label{eq:local-bound-with-endpoint}
\end{equation}
The controlled expansions for $\Y$ and $\widetilde{\Y}$ then imply
\begin{equation*}
\begin{aligned}
&\ \hspace{-2cm}|Y_{t_{i+1}}-\widetilde Y_{t_{i+1}}|+|Y'_{t_{i+1}}-\widetilde Y'_{t_{i+1}}|  \\
\le&\ C_T\left(|Y_{t_i}-\widetilde Y_{t_i}|+|Y'_{t_i}-\widetilde Y'_{t_i}|+d_{\X,\widetilde{\X};\alpha,[t_i,t_{i+1}]}(\Y,\widetilde{\Y})
+\|\X-\widetilde{\X}\|_{\alpha,[t_i,t_{i+1}]}\right).
\end{aligned}
\end{equation*}
Combining this with \eqref{eq:local-bound-with-endpoint} yields \eqref{eq:endpoint-control}. 
Consequently,
\begin{equation}
d_{\X,\widetilde{\X};\alpha,[t_i,t_{i+1}]}(\Y,\widetilde{\Y})\le C_T\mathcal E, \quad i=0,\ldots,N-1 .
\label{eq:local-distance-uniform}
\end{equation}

Finally, Lemma~\ref{lem:patching} patches the local estimates\eqref{eq:local-distance-uniform}, together with the endpoint bound
\eqref{eq:endpoint-control}, into the global controlled distance. 
Hence
\[
d_{\X,\widetilde{\X};\alpha}(\Y,\widetilde{\Y})\le C_T\mathcal E .
\]
That is,
\begin{equation*}
\begin{aligned}
&\ \hspace{-0.7cm}d_{\X,\widetilde{\X};\alpha}(\Y,\widetilde{\Y})\\
\le&\  C_{\al, T}M\Big(T,\|F\|_{C_b^3},|Y'_0|,|\widetilde Y'_0|,|Z'_0|,|\widetilde Z'_0|,\|\Y\|_{\X;\alpha},\|\widetilde{\Y}\|_{\widetilde{\X};\alpha},\|\Z\|_{\X;\alpha},\|\widetilde{\Z}\|_{\widetilde{\X};\alpha},\|\X\|_\alpha,
\|\widetilde{\X}\|_\alpha\Big) \\
&\ \times\Big(|Y_0-\widetilde Y_0|+|Y'_0-\widetilde Y'_0|+|Z_0-\widetilde Z_0|+|Z'_0-\widetilde Z'_0|+
d_{\X,\widetilde{\X};\alpha}(\Z,\widetilde{\Z})+\|\X-\widetilde{\X}\|_\alpha\Big).
\end{aligned}
\end{equation*}
This proves the theorem.
\end{proof}

\noindent
{\bf Acknowledgments.} This work is supported by the National Natural Science Foundation of China (12571019), the Natural Science Foundation of Gansu Province (25JRRA644) and Innovative Fundamental Research Group Project of Gansu Province (23JRRA684).

\noindent
{\bf Declaration of interests. } The authors have no conflicts of interest to disclose.

\end{document}